\documentclass[12pt,reqno]{amsart}
\usepackage{amsmath,amssymb,amsthm,mathrsfs}
\usepackage{setspace}
\usepackage{indentfirst}
\usepackage{enumerate}
\usepackage{float}
\usepackage{cite}
\usepackage{xcolor}
\usepackage{array}
\usepackage{booktabs}
\usepackage{tikz}
\usepackage{pgfplots}
\usepackage[bookmarksnumbered,colorlinks,plainpages]{hyperref}
\usepackage[most]{tcolorbox}
\usetikzlibrary{positioning,calc,fit,backgrounds,decorations.pathreplacing,arrows.meta}

\newtheorem*{theoA}{Theorem A}
\newtheorem*{theoB}{Theorem B}
\newtheorem*{theoC}{Theorem C}
\newtheorem*{theoD}{Theorem D}
\newtheorem*{theoE}{Theorem E}
\newtheorem*{theoF}{Theorem F}
\newtheorem*{theoG}{Theorem G}

\newtheorem*{cor A}{Corollary A}
\newtheorem*{cor B}{Corollary B}

\newtheorem{theo}{Theorem}[section]
\newtheorem{lem}{Lemma}[section]

\newtheorem{defi}{Definition}[section]

\newcommand{\ol}{\overline}
\newcommand{\be}{\begin{equation}}
	\newcommand{\ee}{\end{equation}}
\newcommand{\beas}{\begin{eqnarray*}}
	\newcommand{\eeas}{\end{eqnarray*}}
\newcommand{\bea}{\begin{eqnarray}}
	\newcommand{\eea}{\end{eqnarray}}

\numberwithin{equation}{section}

\begin{document}
\title[O\MakeLowercase {n a class of Pluriharmonic}......]{\LARGE O\Large\MakeLowercase {n a class of Pluriharmonic Mappings in the unit polydisk}}

\date{}
\author[S. M\MakeLowercase{ajumder}, G. H\MakeLowercase{aldar}, A. B\MakeLowercase{anerjee} \MakeLowercase{and} S. P\MakeLowercase{anja}]{S\MakeLowercase{ujoy} M\MakeLowercase{ajumder}$^{1}$, G\MakeLowercase{outam} H\MakeLowercase{aldar}$^2$ $^*$, A\MakeLowercase{bhijit} B\MakeLowercase{anerjee}$^3$ \MakeLowercase{and} S\MakeLowercase{hantanu} P\MakeLowercase{anja}$^4$}

\address{$^{1}$ Sujoy Majumder, Department of Mathematics, Raiganj University, Raiganj, West Bengal-733134, India.}
\email{sm05math@gmail.com, sjm@raiganjuniversity.ac.in}

\address{$^2$ Goutam Haldar, Department of Mathematics, Ghani Khan Choudhury Institute of Engineering and Technology, Narayanpur, Malda-732141, West Bengal, India.}
\email{goutamiitm@gmail.com, goutamiit1986@gmail.com}

\address{$^{3}$ Department of Mathematics, University of Kalyani, West Bengal 741235, India.}
\email{abanerjee\_kal@yahoo.co.in, abhijitbanerjee@klyuniv.ac.in}

\address{$^4$ Shantanu Panja, Department of Mathematics, University of Kalyani, West Bengal 741235, India.}
\email{panjasantu07@gmail.com}

\renewcommand{\thefootnote}{}
\footnote{*\emph{Corresponding Author}: Goutam Halder.}

\renewcommand{\thefootnote}{\arabic{footnote}}
\setcounter{footnote}{0}

\begin{abstract}
	In this paper, we introduce and study the class
	$\mathcal{W}_{\mathcal{H}_n^0}(\alpha)$
	of normalized pluriharmonic mappings, characterized by a suitable bound on their second-order partial derivatives. We establish a one-to-one correspondence between this pluriharmonic class and an associated class of holomorphic functions, thereby extending a result of Ghosh and Vasudevarao \cite{Ghosh-Allu-2019} to the setting of several complex variables. Furthermore, we obtain sharp coefficient bounds, growth estimates and a convex combination theorem for functions in
	$\mathcal{W}_{\mathcal{H}_n^0}(\alpha)$.
	Finally, we introduce sections (partial sums) of pluriharmonic mappings and investigate their properties for functions belonging to
	$\mathcal{W}_{\mathcal{H}_n^0}(\alpha)$.
\end{abstract}
\thanks{Typeset by \AmS -\LaTeX}
\maketitle

\section{\bf Introduction}
The study of pluriharmonic mappings lies at the intersection of complex analysis, complex geometry and geometric function theory. These mappings preserve many of the important geometric properties of holomorphic functions while offering greater flexibility through their co-holomorphic part. This makes them a useful tool for extending many classical results, such as coefficient estimates, distortion theorems, univalence criteria and Landau--Bloch type theorems, to the setting of several complex variables. In addition, pluriharmonic mappings have important applications in the study of minimal surfaces, complex manifolds, elliptic partial differential equations and quasiconformal mappings. Because of their rich geometric structure and wide range of applications, they continue to attract considerable attention in modern complex analysis.\vspace{1.2mm}

In modern theoretical physics (especially string theory, general relativity, and quantum field theory), physical configurations are often modeled as geometric maps between manifolds. When these manifolds carry complex structures, pluriharmonic mappings and their geometric counterpart, pluriharmonic maps, serve as the foundational mathematical language. Pluriharmonic maps act as critical points of the energy functional on these complex domains. Because pluriharmonic functions can be locally split into the sum of a holomorphic and an anti-holomorphic function ($f = h + \bar{g}$), they allow physicists to classify stable, minimum-energy field configurations (solitons or instantons) in higher dimensions.\vspace{1.2mm}

\subsection{\bf {Basic ideas of pluriharmonic mapping}}
We write $z_j = x_j + i y_j\;(i^2 = -1,\; j = 1,\dots,n)$,
where $x_j$ and $y_j$ are real numbers. We set
$f(z) = u(x,y) + i v(x,y)$,
where $u(x,y)$ and $v(x,y)$ are the real and imaginary parts of $f(z)$; $x = (x_1,\dots,x_n)$ and $y = (y_1,\dots,y_n)$.
The Cauchy--Riemann equations for each $z_j\;(j=1,\dots,n)$ are
\begin{align}\label{Eq 1.1}
	\frac{\partial u}{\partial x_j}
	=
	\frac{\partial v}{\partial y_j}\quad \text{and}\quad
	\frac{\partial u}{\partial y_j}
	=
	-\,\frac{\partial v}{\partial x_j}
	\qquad (j=1,\dots,n).
\end{align}

By differentiating (\ref{Eq 1.1}) with respect to $x_k$ and $y_k$, we see that both $u$ and $v$ satisfy the
following system of partial differential equations of second order:
\begin{align}\label{Eq 1.2}
	\frac{\partial^2}{\partial x_j \partial x_k}
	+
	\frac{\partial^2}{\partial y_j \partial y_k}
	= 0
	\quad \text{and} \quad
	\frac{\partial^2}{\partial x_j \partial y_k}
	-
	\frac{\partial^2}{\partial x_k \partial y_j}
	= 0
	\qquad (j,k = 1,\dots,n).
\end{align}

For a complex variable $z_j = x_j + i y_j$, we define
\begin{align}\label{Eq 1.3}
	\frac{\partial}{\partial z_j}
	= \frac{1}{2}\left( \frac{\partial}{\partial x_j}
	- i \frac{\partial}{\partial y_j} \right),\quad
	\frac{\partial}{\partial \bar z_j}
	= \frac{1}{2}\left( \frac{\partial}{\partial x_j}
	+ i \frac{\partial}{\partial y_j} \right).
\end{align}
\begin{align*}
	\partial =\sum\limits_{j=1}^n \frac{\partial }{\partial z_j} d z_j,\quad \ol{\partial} =\sum\limits_{j=1}^n \frac{\partial }{\partial \ol z_j} d \ol {z}_j \quad \text{and} \quad d=\partial+\ol{\partial}.
\end{align*}

We know that a function $f$ defined on an open subset $U\subset \mathbb{R}^n$ is said to be of $C^k$-class if $f$ is $k$-times continuously differentiable.

Let $f(z) = u(x,y) + i v(x,y)\;(x = (x_1,\dots,x_n), y = (y_1,\dots,y_n))$, where both $u$ and $v$ are of $C^2$-class. A direct calculation on (\ref{Eq 1.3}) shows that 
\begin{align}\label{Eq 1.5}
	4\frac{\partial^2 f(z)}{\partial \bar z_j \partial z_k}=&\frac{\partial^2 u(x,y)}{\partial x_j x_k}+\frac{\partial^2 u(x,y)}{\partial y_j y_k}+i\left(\frac{\partial^2 v(x,y)}{\partial x_j x_k}+\frac{\partial^2 v(x,y)}{\partial y_j y_k}\right)\\&-i\left(\frac{\partial^2 u(x,y)}{\partial x_j y_k}-\frac{\partial^2 u(x,y)}{\partial x_k y_j}\right)+\left(\frac{\partial^2 v(x,y)}{\partial x_j y_k}-\frac{\partial^2 v(x,y)}{\partial x_k y_j}\right).\nonumber
\end{align}

A function $f:\Omega \to \mathbb{C}$ defined on an open set $\Omega\subset \mathbb{C}^n$ is said to be holomorphic if $f$ is of $C^1$-class and satisfies
\begin{align}\label{Eq 1.4}
\ol{\partial} f=0,\;\; \text{i.e.},\;\; \frac{\partial f(z)}{\partial \bar z_j}=0\quad \text{on}\;\;\Omega\;\;\text{ for all j}.
\end{align}

\subsection*{{\bf Pluriharmonic mapping}}
A real-valued function $\phi(x,y)$, where $x = (x_1,\dots,x_n)$ and $y = (y_1,\dots,y_n)$ is \emph{pluriharmonic} if it satisfies the conditions  (\ref{Eq 1.2}). Thus a continuous complex-valued function $f(z)=u(x,y)+iv(x,y)$, where
$x = (x_1,\dots,x_n)$ and $y = (y_1,\dots,y_n)$ is a complex-valued \emph{pluriharmonic} function in a domain $\Omega\subset \mathbb{C}^n$, if both $u(x,y)$ and $v(x,y)$ are real-valued \emph{pluriharmonic} functions in $\Omega$.
If $u(x,y)$ and $v(x,y)$, where $x = (x_1,\dots,x_n)$ and $y = (y_1,\dots,y_n)$ satisfy (\ref{Eq 1.1}), then we call $v(x,y)$, a \emph{pluriharmonic conjugate} of $u(x,y)$.

\par Thus for functions $f(z)=u(x,y)+iv(x,y)$, where
$x = (x_1,\dots,x_n)$ and $y = (y_1,\dots,y_n)$ with continuous second order partial derivatives, it is clear from (\ref{Eq 1.5}) and (\ref{Eq 1.4}) that $\frac{\partial f(z)}{\partial z_j}$ is holomorphic on $\Omega$ for all $j$ if  $f(z)$ is pluriharmonic function.
\vspace{1.5mm}
\par In a simply connected domain $\Omega\subset \mathbb{C}^n$, let $f(z)$ be a complex-valued pluriharmonic function. We recall that $\frac{\partial f(z)}{\partial z_j}$ is holomorphic on $\Omega$ for all $j=1,2,\ldots,n$ if $f(z)$ is pluriharmonic and let $\frac{\partial h(z)}{\partial z_j}=\frac{\partial f(z)}{\partial z_j}\;(j=1,2,\ldots,n)$, where $h(z)$ is holomorphic in $\Omega$. Now let $g(z)=\ol {f(z)}-\ol {h(z)}$ and we observe that
\[\frac{\partial g(z)}{\partial \ol{z_j}}=\ol{\frac{\partial f(z)}{\partial z_j}}-\ol{\frac{\partial h(z)}{\partial z_j}}=0\quad \text{in}\;\Omega,\quad j=1,2,\ldots,n\]
by the definition of $h$. Thus $g(z)$ is holomorphic in $\Omega$. Therefore the pluriharmonic function $f(z)$ has the representation $f(z)=h(z)+\ol{g(z)}$, where $h(z)$ and $g(z)$ are holomorphic in $\Omega$. \vspace{1.2mm}
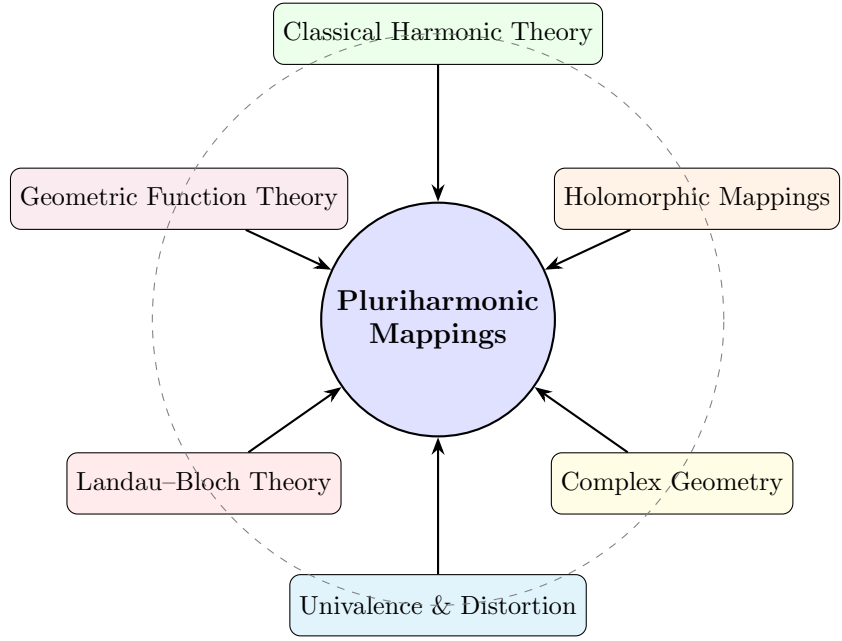
\begin{figure}[t]
	\centering
	
	\begin{tikzpicture}[
		scale=0.9,          
		transform shape,    
		every node/.style={font=\small},
		>=Stealth
		]

		\node[
		circle,
		draw,
		thick,
		fill=blue!12,
		minimum size=2.8cm,
		align=center,
		font=\bfseries
		] (P) {Pluriharmonic\\Mappings};
		
		\node[
		draw,
		rounded corners,
		fill=green!8,
		minimum width=3.3cm,
		minimum height=0.9cm
		] (H) at (90:4.2) {Classical Harmonic Theory};
		
		\node[
		draw,
		rounded corners,
		fill=orange!10,
		minimum width=3.2cm,
		minimum height=0.9cm
		] (A) at (25:4.2) {Holomorphic Mappings};
		
		\node[
		draw,
		rounded corners,
		fill=yellow!12,
		minimum width=3.3cm,
		minimum height=0.9cm
		] (G) at (-35:4.2) {Complex Geometry};
		
		\node[
		draw,
		rounded corners,
		fill=cyan!10,
		minimum width=3.6cm,
		minimum height=0.9cm
		] (U) at (-90:4.2) {Univalence \& Distortion};
		
		\node[
		draw,
		rounded corners,
		fill=red!8,
		minimum width=3.7cm,
		minimum height=0.9cm
		] (L) at (-145:4.2) {Landau--Bloch Theory};
		
		\node[
		draw,
		rounded corners,
		fill=purple!8,
		minimum width=3.6cm,
		minimum height=0.9cm
		] (F) at (155:4.2) {Geometric Function Theory};

		\draw[->,thick] (H)--(P);
		
		\draw[->,thick] (A)--(P);
		
		\draw[->,thick] (G)--(P);
		
		\draw[->,thick] (U)--(P);
		
		\draw[->,thick] (L)--(P);
		
		\draw[->,thick] (F)--(P);
		
		
		\draw[dashed,gray]
		(0,0) circle (4.2);
		
	\end{tikzpicture}
	
	\caption{
		Position of pluriharmonic mappings within several complex variables.
	}
	
	\label{Fig:ConceptMap}
	
\end{figure}\par 
In one complex variable, harmonic functions are closely related to analytic functions, since every harmonic function can locally be expressed as the real part of an analytic function. However, when one moves to higher dimensions, ordinary harmonicity is no longer sufficient to capture the rich complex structure of $\mathbb{C}^n$. This limitation motivates the introduction of pluriharmonic functions, namely functions whose restrictions to every complex line are harmonic. Pluriharmonicity thus preserves the intrinsic geometry of complex spaces and provides a natural framework for extending many classical results of geometric function theory to several complex variables.

Figure~\ref{Fig:ConceptMap} highlights the central role of pluriharmonic mappings in several complex variables. They provide a natural bridge between classical harmonic analysis, holomorphic mappings, complex geometry and geometric function theory, thereby serving as an effective framework for extending many fundamental results from one complex variable to higher dimensions.\par 
From the perspective of geometric function theory, pluriharmonic functions constitute an important class of mappings that preserve many of the geometric features of holomorphic functions while offering greater flexibility. In particular, mappings of the form $f(z)=h(z)+\ol{g(z)}$,
where $h$ and $g$ are holomorphic functions of several complex variables, are pluriharmonic. Such mappings have attracted significant attention in recent years due to their wide-ranging applications in the study of univalence, distortion and covering theorems, Schwarz-type lemmas, and Landau-Bloch-type results in higher-dimensional complex spaces.

\smallskip
Pluriharmonic functions are also closely related to pseudoconvex domains, which are fundamental objects in the theory of several complex variables. Since every pluriharmonic function is both plurisubharmonic and plurisuperharmonic, it represents the \emph{flat} case of complex potential theory. This connection provides valuable geometric and analytic tools for investigating the structure of complex domains and the behavior of holomorphic and harmonic mappings.

\smallskip
Consequently, pluriharmonic functions occupy a central position in modern geometric function theory. They extend the notion of harmonicity to several complex variables, bridge the analytic and geometric aspects of complex mappings, and furnish a natural framework for generalizing many classical results from the complex plane to higher-dimensional complex spaces. Their study remains an active area of research, with significant developments in the theory of univalent mappings, invariant metrics, complex geometry, and multidimensional harmonic analysis.

\subsection{{\bf Basic Notations in several complex variables}}
Let $\mathbb{P}\Delta(a;r)=\lbrace z\in\mathbb{C}^n: |z_i-a_i|<r_i,\;i=1,2,\ldots,n\rbrace$ denote the open polydisc in $\mathbb{C}^n$. Hrere $a=(a_1,\ldots,a_n)\in\mathbb{C}^n$ is called the centre of the polydisc and $r=(r_1,r_2,\ldots,r_n)\in\mathbb{R}^n\;(r_i>0)$ is called the polyradius. In particular $\mathbb{P}\Delta(0;1)=\mathbb{P}\Delta(0_n;1_n)$, where $0_n=(0,0,\ldots,0)$ and $1_n=(1,1,\ldots,1)$. This polydisc is called the unit polydisc in $\mathbb{C}^n$. The unit disk in the complex plane is denoted by $\mathbb{D}$.

The absolute value of a complex number $z_1$ is denoted by $|z_1|$ and for $z=(z_1,\ldots,z_n)\in\mathbb{C}^n$, we define
\begin{align*}
	||z||^2=\sum\limits_{k=1}^n|z_k|^2\quad \text{and}\quad\|z\|_{\infty}=\max\limits_{1\leq i\leq n}|z_i|.
\end{align*}

\smallskip
A multi-index $\beta=(\beta_1,\ldots,\beta_n)$ of dimension $n$ consists of n non-negative integers $\beta_j,\;1\leq j\leq n$; the degree of a multi-index $\beta$ is the sum $|\beta|=\sum_{j=1}^n \beta_j$ and we denote $\beta!=\beta_1!\ldots \beta_n!$. We denote by $\mathbb{Z}[1,n]$ the set $\{1,2,\ldots,n\}$.  For $z=(z_1,\ldots,z_n)\in\mathbb{C}^n$ and a multi-index $\beta=(\beta_1,\ldots,\beta_n)$, we define 
\begin{align*}
	z^{\beta}=\prod\limits_{j=1}^n z_j^{\beta_j}\;\;\text{and}\;\;|z|^{\beta}=\prod\limits_{j=1}^n |z_j|^{\beta_j}.
	\end{align*}

Let $f(z)$ be a holomorphic function in a domain $\Omega\subset \mathbb{C}^n$. Then in the polydisk $\mathbb{P}\Delta(0;r)\subset \Omega$, $f(z)$ has a power series expansion in $z_1,\ldots,z_n$,
\begin{align*}
f(z)=\sum\limits_{\beta_1,\beta_2,\ldots,\beta_n=0}^{\infty} a_{\beta_1,\beta_2,\ldots,\beta_n}z_1^{\beta_1}z_2^{\beta_2}\ldots z_n^{\beta_n}=
\sum\limits_{|\beta|=0}^{\infty} a_{\beta}z^{\beta}=\sum\limits_{|\beta|=0}^{\infty} P_{|\beta|}(z),
\end{align*}
which is absolutely convergent in $\mathbb{P}\Delta(0;r)$, where the term $P_k(z)$ is a homogeneous polynomial of degree $k$. 

\subsection{{\bf Different classes of pluriharmonic mappings}}
Let $\mathcal{H}_n$ denote the class of complex-valued pluriharmonic functions $f$ in $\mathbb{P} \Delta(0;1)$, normalized by 
\begin{align*}
f(0)= 0\; \mbox{and}\;\nabla f(0):=\left(\frac{\partial f(0)}{\partial z_1},\ldots,\frac{\partial f(0)}{\partial z_n}\right)=(1,1,\ldots,1).
\end{align*} 

Obviously when $\dim(\mathbb{C}^n)=1$, $\mathcal{H}(=\mathcal{H}_1)$ is the class of complex-valued harmonic functions $f$ in the unit disk $\mathbb{D}$, normalized by $f(0)=0$, $\frac{\partial f(0)}{\partial z}=1$.
Each function $f$ in $\mathcal{H}_n$ can be expressed as $f=h+\ol g$, where $h$ and $g$ are holomorphic functions in $\mathbb{P} \Delta(0;1)$. Here $h$ and $g$ are called the holomorphic and co-holomorphic parts of $f$ respectively, and have power series representations
\begin{align*}
h(z)=\sum\limits_{j=1}^n z_j+\sum\limits_{k=2}^{\infty}\sum\limits_{|\beta|=k} a_{\beta} z^{\beta}\quad {and}\quad g(z)=\sum\limits_{k=1}^{\infty}\sum\limits_{|\beta|=k} b_{\beta} z^{\beta}.
\end{align*}

If $g(z)\equiv 0$, then $\mathcal{H}_n$ reduces to the class $\mathcal{A}_n$ of holomorphic functions in the unit polydisk $\mathbb{P} \Delta(0;1)$ with $f(0)= 0$ and $\nabla f(0)= (1, 1, \ldots, 1).$ Clearly when $\dim(\mathbb{C}^n)=1$, $\mathcal{A}(=\mathcal{A}_1)$ is the class of holomorphic functions $f$ in the unit disk $\mathbb{D}$, normalized by $f(0)=0=\frac{\partial f(0)}{\partial z}-1$.
Let
\begin{align*}
\mathcal{H}^0_n=\left\lbrace f\in\mathcal{H}_n: \ol{\partial} f(0)=0\right\rbrace.
\end{align*}

Hence for any function $f=h+\ol g$ in $\mathcal{H}^0_n$, its holomorphic and co-holomorphic parts can be represented by
\begin{align}\label{Eq 1.7}
h(z)=\sum\limits_{j=1}^n z_j+\sum\limits_{k=2}^{\infty}\sum\limits_{|\beta|=k} a_{\beta} z^{\beta}\quad {and} \quad g(z)=\sum\limits_{k=2}^{\infty}\sum\limits_{|\beta|=k} b_{\beta} z^{\beta}
\end{align}
respectively. 
Let $\mathcal{S}_{\mathcal{H}}$ be the subclass of $\mathcal{H}$ consisting of univalent and sense preserving harmonic mappings.

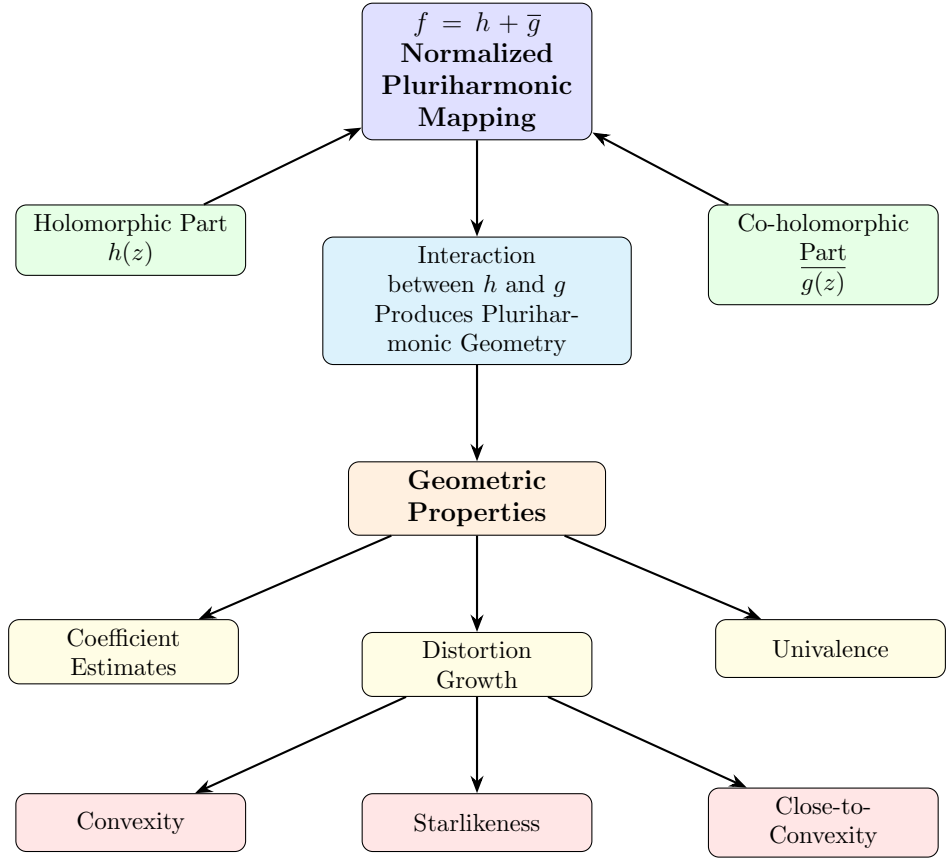
\begin{figure}[t]
	\centering
	
	\begin{tikzpicture}[ scale=0.85,
		transform shape,
		>=Stealth,
		node distance=1.6cm,
		every node/.style={
			font=\small,
			align=center
		},
		box/.style={
			draw,
			rounded corners=4pt,
			text width=3.3cm,
			minimum height=.9cm,
			align=center
		},
		main/.style={
			box,
			fill=blue!12,
			font=\bfseries
		},
		analytic/.style={
			box,
			fill=green!10
		},
		interaction/.style={
			box,
			fill=cyan!12,
			minimum width=4.8cm
		},
		geometry/.style={
			box,
			fill=orange!12,
			font=\bfseries,
			minimum width=4cm
		},
		result/.style={
			box,
			fill=yellow!12,
			minimum width=3.3cm
		},
		class/.style={
			box,
			fill=red!10,
			minimum width=2.9cm
		}
		]

		\node[main] (F)
		{$f=h+\overline{g}$\\Normalized Pluriharmonic Mapping};

		\node[analytic,below left=1.0cm and 1.8cm of F] (H)
		{Holomorphic Part\\$h(z)$};
		
		\node[analytic,below right=1.0cm and 1.8cm of F] (G)
		{Co-holomorphic Part\\$\overline{g(z)}$};

		\node[interaction,below=1.5cm of F] (I)
		{Interaction between $h$ and $g$\\Produces Pluriharmonic Geometry};

		\node[geometry,below=1.5cm of I] (P)
		{Geometric Properties};

		\node[result,below left=1.3cm and 1.7cm of P] (C)
		{Coefficient\\Estimates};
		
		\node[result,below=1.5cm of P] (D)
		{Distortion\\Growth};
		
		\node[result,below right=1.3cm and 1.7cm of P] (U)
		{Univalence};

		\node[class,below=1.5cm of D] (S)
		{Starlikeness};
		
		\node[class,left=1.8cm of S] (CV)
		{Convexity};
		
		\node[class,right=1.8cm of S] (CC)
		{Close-to-\\Convexity};

		\draw[->,thick] (H)--(F);
		\draw[->,thick] (G)--(F);
		
		\draw[->,thick] (F)--(I);
		
		\draw[->,thick] (I)--(P);
		
		\draw[->,thick] (P)--(C);
		\draw[->,thick] (P)--(D);
		\draw[->,thick] (P)--(U);
		
		\draw[->,thick] (D)--(S);
		\draw[->,thick] (D)--(CV);
		\draw[->,thick] (D)--(CC);
		
	\end{tikzpicture}
	
	\caption{
		Conceptual architecture of the theory developed in this paper.
			}
	
	\label{fig:conceptual_architecture}
	
\end{figure}

A complex-valued pluriharmonic mapping $f\in \mathcal{H}_n$ is said to be starlike if $f(\mathbb{P} \Delta(0;1))$ is a starlike
domain with respect to the origin. We denote the class of pluriharmonic starlike functions in $\mathbb{P} \Delta(0;1)$ by $\mathcal{S}^*_{\mathcal{H}_n}$. A function $f$ in $\mathcal{H}_n$ is said to be convex if $f(\mathbb{P} \Delta(0;1))$ is convex. We denote the class of pluriharmonic convex mappings in $\mathbb{P} \Delta(0;1)$ by $\mathcal{K}_{\mathcal{H}_n}$. A function $f\in\mathcal{H}_n$ is said to be close-to-convex if $f(\mathbb{P} \Delta(0;1))$ is a close-to-convex domain. We denote the class of pluriharmonic close-to-convex mappings in $\mathbb{P} \Delta(0;1)$ by $\mathcal{C}_{\mathcal{H}_n}$. Finally, we denote by ${\mathcal{S}^*_{\mathcal{H}_n}}^0$, $\mathcal{K}_{\mathcal{H}_n}^0$ and $\mathcal{C}_{\mathcal{H}_n}^0$, the
subclasses of $\mathcal{S}^*_{\mathcal{H}_n}$, $\mathcal{K}_{\mathcal{H}_n}$ and $\mathcal{C}_{\mathcal{H}_n}$ with $\ol{\partial}f(0)=(0,0,\ldots,0)$ respectively.\par 
In 1984, Clunie and Sheil-Small \cite{Clunie-Sheil-Small-1984} investigated the class $\mathcal{S}_{\mathcal{H}}$, together with some geometric subclasses. Subsequently, the class $\mathcal{S}_{\mathcal{H}}$ and its subclasses ${\mathcal{S}^*_{\mathcal{H}_n}}^0$, $\mathcal{K}_{\mathcal{H}_n}^0$, $\mathcal{C}_{\mathcal{H}_n}^0$, $\mathcal{S}^*_{\mathcal{H}_n}$, $\mathcal{K}_{\mathcal{H}_n}$ and $\mathcal{C}_{\mathcal{H}_n}$, have been extensively studied by several authors (see \cite{Aizenberg-Aytuna-Djakov-JMAA-2001}-\cite{Ghosh-Allu-2019}, \cite{Hernandez-Martin-2013}-\cite{Liu-Yang-2019},  \cite{Muhanna-CVEE-2010}, \cite{Ponnusamy-Allu-Vuorinen-2009}-\cite{Ponnusamy-Kaliraj-Starkov-2017}).
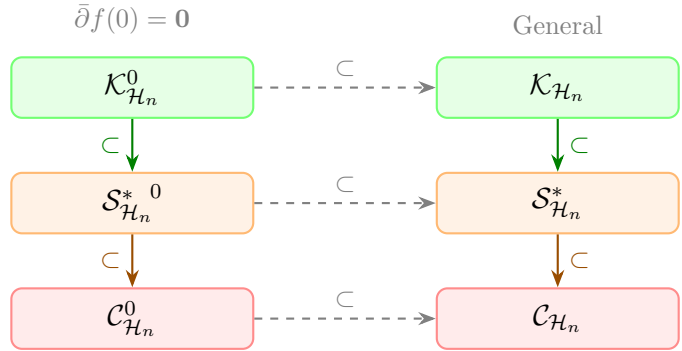
\begin{figure}[ht]
	\centering
	\begin{tikzpicture}[
		>=Stealth,
		every node/.style={font=\small},
		cls/.style={
			draw, rounded corners=4pt, thick,
			minimum width=3.2cm, minimum height=0.8cm,
			align=center
		}
		]

		\node[cls, fill=green!10, draw=green!55] (Kz)
		{$\mathcal{K}^0_{\mathcal{H}_n}$};
		\node[cls, fill=orange!10, draw=orange!55,
		below=0.7cm of Kz] (Sz)
		{${\mathcal{S}^*_{\mathcal{H}_n}}^{\!0}$};
		\node[cls, fill=red!8, draw=red!45,
		below=0.7cm of Sz] (Cz)
		{$\mathcal{C}^0_{\mathcal{H}_n}$};

		\node[cls, fill=green!10, draw=green!55,
		right=2.4cm of Kz] (K)
		{$\mathcal{K}_{\mathcal{H}_n}$};
		\node[cls, fill=orange!10, draw=orange!55,
		right=2.4cm of Sz] (S)
		{$\mathcal{S}^*_{\mathcal{H}_n}$};
		\node[cls, fill=red!8, draw=red!45,
		right=2.4cm of Cz] (C)
		{$\mathcal{C}_{\mathcal{H}_n}$};

		\draw[->, thick, green!50!black]
		(Kz) -- node[left,font=\footnotesize]{$\subset$} (Sz);
		\draw[->, thick, orange!60!black]
		(Sz) -- node[left,font=\footnotesize]{$\subset$} (Cz);
		\draw[->, thick, green!50!black]
		(K)  -- node[right,font=\footnotesize]{$\subset$} (S);
		\draw[->, thick, orange!60!black]
		(S)  -- node[right,font=\footnotesize]{$\subset$} (C);
		
		\draw[->, thick, dashed, gray]
		(Kz) -- node[above,font=\footnotesize]{$\subset$} (K);
		\draw[->, thick, dashed, gray]
		(Sz) -- node[above,font=\footnotesize]{$\subset$} (S);
		\draw[->, thick, dashed, gray]
		(Cz) -- node[above,font=\footnotesize]{$\subset$} (C);

		\node[font=\footnotesize, text=gray,
		above=0.15cm of Kz]{$\bar\partial f(0)=\mathbf{0}$};
		\node[font=\footnotesize, text=gray,
		above=0.15cm of K]{General};
		
	\end{tikzpicture}
	\caption{Inclusion relations among the geometric subclasses.
		Solid arrows denote containment within a column; dashed arrows
		indicate that the superscript-$0$ class is a proper subclass of
		the corresponding general class.
		}
	\label{fig:inclusions}
\end{figure}

In 2013, Li and Ponnusamy \cite{Li-Ponnusamy-2013} investigated the properties of functions in $\mathcal{P}_{\mathcal{H}}^{0}$ given by
\begin{align*}
\mathcal{P}_{\mathcal{H}}^{0}
:= \left\{
f=h+g \in \mathcal{H} :
\Re \bigl(h'(z)\bigr) > |g'(z)|,\;\; z\in\mathbb{D}
\right\}.
\end{align*}

The class $\mathcal{P}_{\mathcal{H}}^{0}$ is closely related to the class
\begin{align*}
\mathcal{R}
:= \left\{
f\in\mathcal{A} :
\Re\bigl(f'(z)\bigr)>0,\;\; z\in\mathbb{D}
\right\},
\end{align*}
introduced by MacGregor \cite{MacGregor-1962}. It is known that a harmonic function $f=h+\ol g$ belongs to
$\mathcal{P}_{\mathcal{H}}^{0}$ if and only if the analytic functions
$h+\lambda g$ belong to $\mathcal{R}$ for each $\lambda$
($|\lambda|=1$) (see \cite{Li-Ponnusamy-2013,Li-Ponnusamy-2013a}).
Using this property, Li and Ponnusamy \cite{Li-Ponnusamy-2013} obtained coefficient bounds and the radius of convexity for functions in $\mathcal{P}_{\mathcal{H}}^{0}$.

\begin{table}[ht]
	\centering
	\small
	\renewcommand{\arraystretch}{1.5}
	\setlength{\tabcolsep}{3.5pt}
	
	\begin{tabular}{>{\centering\arraybackslash}p{2.2cm}
			>{\raggedright\arraybackslash}p{3.4cm}
			>{\raggedright\arraybackslash}p{5.3cm}
			>{\raggedright\arraybackslash}p{3.8cm}}
		\toprule
		
		\textbf{Class}
		&
		\textbf{Role}
		&
		\textbf{Definition}
		&
		\textbf{Normalization / Relation}
		\\
		
		\midrule
		\multicolumn{4}{c}{\textbf{Fundamental Classes}}
		\\
		\midrule
		
		$\mathcal H_n$
		&
		Normalized pluriharmonic mappings
		&
		$f$ is pluriharmonic on $\mathbb P\Delta(0;1)$.
		&
		$f(0)=0,\;
		\nabla f(0)=(1,\ldots,1)$.
		\\[1mm]
		
		$\mathcal H_n^{0}$
		&
		Normalized subclass
		&
		$f\in\mathcal H_n$ with
		$\overline{\nabla}f(0)=\mathbf0$.
		&
		Equivalent to
		$g(z)=\displaystyle\sum_{|\alpha|\ge2}b_\alpha z^\alpha$.
		\\[1mm]
		
		$\mathcal A_n$
		&
		Normalized holomorphic subclass
		&
		$f\in\mathcal H_n$ with
		$g\equiv0$.
		&
		$\mathcal A_n\subset\mathcal H_n^{0}\subset\mathcal H_n$.
		\\
		
		\midrule
		\multicolumn{4}{c}{\textbf{Geometric Subclasses}}
		\\
		\midrule
		
		$\mathcal S_{\mathcal H_n}^{*}$
		&
		Pluriharmonic starlike mappings
		&
		$f(\mathbb P\Delta(0;1))$ is a starlike domain
		with respect to the origin.
		&
		Geometric subclass of
		$\mathcal H_n$.
		\\[1mm]
		
		$\mathcal K_{\mathcal H_n}$
		&
		Pluriharmonic convex mappings
		&
		$f(\mathbb P\Delta(0;1))$ is a convex domain.
		&
		Geometric subclass of
		$\mathcal H_n$.
		\\[1mm]
		
		$\mathcal C_{\mathcal H_n}$
		&
		Pluriharmonic close-to-convex mappings
		&
		$f(\mathbb P\Delta(0;1))$ is a close-to-convex domain.
		&
		Geometric subclass of
		$\mathcal H_n$.
		\\
		
		\midrule
		\multicolumn{4}{c}{\textbf{Normalized Geometric Subclasses}}
		\\
		\midrule
		
		${\mathcal S_{\mathcal H_n}^{*}}^{0}$
		&
		Normalized starlike mappings
		&
		$f\in\mathcal S_{\mathcal H_n}^{*}$ with
		$\overline{\nabla}f(0)=\mathbf0$.
		&
		Subclass of
		$\mathcal H_n^{0}$.
		\\[1mm]
		
		$\mathcal K_{\mathcal H_n}^{0}$
		&
		Normalized convex mappings
		&
		$f\in\mathcal K_{\mathcal H_n}$ with
		$\overline{\nabla}f(0)=\mathbf0$.
		&
		Subclass of
		$\mathcal H_n^{0}$.
		\\[1mm]
		
		$\mathcal C_{\mathcal H_n}^{0}$
		&
		Normalized close-to-convex mappings
		&
		$f\in\mathcal C_{\mathcal H_n}$ with
		$\overline{\nabla}f(0)=\mathbf0$.
		&
		Subclass of
		$\mathcal H_n^{0}$.
		\\
		
		\bottomrule
	\end{tabular}
	\vspace{.5cc}
	\caption{Hierarchy of the principal classes of pluriharmonic mappings considered in this paper. The superscript $0$ indicates the additional normalization
		$\overline{\nabla}f(0)=\mathbf0$, equivalently, the co-holomorphic part
		$g$ vanishes to second order at the origin. }
	
	\label{tab:HierarchyClasses}
\end{table}
In 1977, Chichra \cite{Chichra-1977} studied the following class
\begin{align*}
\mathcal{W}(\alpha)
=
\left\{
f \in \mathcal{A} :
\Re\!\left(f'(z)+\alpha z f''(z)\right)>0,
\quad z\in\mathbb{D},\ \alpha\geq 0
\right\}.
\end{align*}

In 2010, the regions of variability for functions in $\mathcal{W}(\alpha)$ were studied by Ponnusamy and Vasudevarao \cite{Ponnusamy-Vasudevarao-2010}, and in 1982, Singh and Singh \cite{Singh-Singh-1982} showed that $\mathcal{W}(1)$ is a subclass of $\mathcal{S}^{*}$.

In 2014, Nagpal and Ravichandran \cite{Nagpal-Ravichandran-2014} studied the following class
\begin{align*}
\mathcal{W}_{\mathcal{H}}^{0}
:=
\left\{
f=h+g \in \mathcal{H} :
\Re\!\left(h'(z)+z h''(z)\right)
>
\left|g'(z)+z g''(z)\right|,
\quad z\in\mathbb{D}
\right\},
\end{align*}
which is the harmonic analogue of $\mathcal{W}(1)$.

It is known that $\mathcal{W}_{\mathcal{H}}^{0}$ is a subclass of both
$\mathcal{S}_{\mathcal{H}}^{*0}$ and $\mathcal{P}_{\mathcal{H}}^{0}$.
In particular, members of $\mathcal{W}_{\mathcal{H}}^{0}$ are fully starlike in $\mathbb{D}$.
Sharp coefficient bounds and growth theorems for functions in
$\mathcal{W}_{\mathcal{H}}^{0}$ have been obtained in \cite{Nagpal-Ravichandran-2014}.

\section{\bf {Preliminary results}}	
In $2019$, Ghosh and Vasudevarao \cite{Ghosh-Allu-2019} have introduced an extension of the class $\mathcal{W}^0_{\mathcal{H}}$ as follows.

For $\alpha\geq 0$, let the class $\mathcal{W}^0_{\mathcal{H}}(\alpha)$ be defined by
\begin{align*}
\mathcal{W}^0_{\mathcal{H}}(\alpha)=\left\lbrace f=h+\ol g\in\mathcal{H}:\Re(h'(z)+\alpha zh''(z))>|g'(z)+\alpha zg''(z)|\;\;\text{for}\;z\in\mathbb{D} \right\rbrace
\end{align*}
with $g'(0)=0$.\vspace{1.2mm}

In the same paper, Ghosh and Vasudevarao \cite{Ghosh-Allu-2019} first established a result providing a one-to-one correspondence between the classes $\mathcal{W}_{\mathcal{H}}^0(\alpha)$ and $\mathcal{W}(\alpha)$.

\begin{theoA}\emph{\cite[Theorem 4.1]{Ghosh-Allu-2019}} The harmonic mapping $f=h+\ol g$ belongs to $\mathcal{W}_{\mathcal{H}}^0(\alpha)$ if and only if the analytic function $F=h+\varepsilon g$ belongs to $\mathcal{W}(\alpha)$ for each $|\varepsilon|=1$.
\end{theoA}

The following results due to Ghosh and Vasudevarao \cite{Ghosh-Allu-2019} provides sharp coefficient bounds for functions in $\mathcal{W}_{\mathcal{H}}^0(\alpha)$.

\begin{theoB}\emph{\cite[Theorem 4.2]{Ghosh-Allu-2019}}
Let $f=h+\overline{g}\in \mathcal{W}_{\mathcal{H}}^{0}(\alpha)$ for
$\alpha\geq 0$ be of the form 
\begin{align}\label{1.2}
h(z)=z+\sum\limits_{m=2}^{\infty}a_m z^m\;\;\text{and}\;\;g(z)=\sum\limits_{m=2}^{\infty} b_m z^m.
\end{align}
Then for $m\geq 2$,
\begin{align*}
|b_m|\leq \frac{1}{\alpha m^2+m(1-\alpha)}.
\end{align*}
	
The result is sharp when $f$ is given by
\begin{align*}
f(z)=z+\frac{1}{\alpha m^2+m(1-\alpha)}\,\overline{z}^{\,m}.
\end{align*}
\end{theoB}

\begin{theoC}\emph{\cite[Theorem 4.3]{Ghosh-Allu-2019}}
Let $f=h+\overline{g}\in \mathcal{W}_{\mathcal{H}}^{0}(\alpha)$ for
$\alpha\geq 0$ and be of the form \eqref{1.2}. Then for $m\geq 2$, we have
\begin{enumerate}
\item[\emph{(i)}]
$|a_m|+|b_m|
\leq \frac{2}{\alpha m^2+m(1-\alpha)}$.
\item[\emph{(ii)}]
$\bigl||a_m|-|b_m|\bigr|\leq
\frac{2}{\alpha m^2+m(1-\alpha)}$.
\item[\emph{(iii)}]
$|a_m|\leq
\frac{2}{\alpha m^2+m(1-\alpha)}$.
	\end{enumerate}
	
All these results are sharp for the function
\begin{align*}
f(z)=z+\sum_{m=2}^{\infty}
\frac{2}{\alpha m^2+m(1-\alpha)}\,z^m.
\end{align*}
\end{theoC}

In the same paper, Ghosh and Vasudevarao \cite{Ghosh-Allu-2019} provides the growth estimate for functions in $\mathcal{W}_{\mathcal{H}}^0(\alpha)$.

\begin{theoD}\emph{\cite[Theorem 4.4]{Ghosh-Allu-2019}} Let $f=h+\ol g\in \mathcal{W}_{\mathcal{H}}^0(\alpha)$ be given by \eqref{1.2} with $0<\alpha\le 1$. Then
\begin{align*}
|z|+2\sum\limits_{m=2}^{\infty}\frac{(-1)^{m-1}|z|^m}{\alpha m^2+m(1-\alpha)}\leq |f(z)|\leq |z|+2\sum\limits_{m=2}^{\infty}\frac{|z|^m}{\alpha m^2+m(1-\alpha)}.
\end{align*}
	
Both inequalities are sharp when $f$ is given by
$f(z)=z+\sum\limits_{m=2}^{\infty}\frac{2}{\alpha m^2+m(1-\alpha)}z^n$, or its rotations.
\end{theoD}
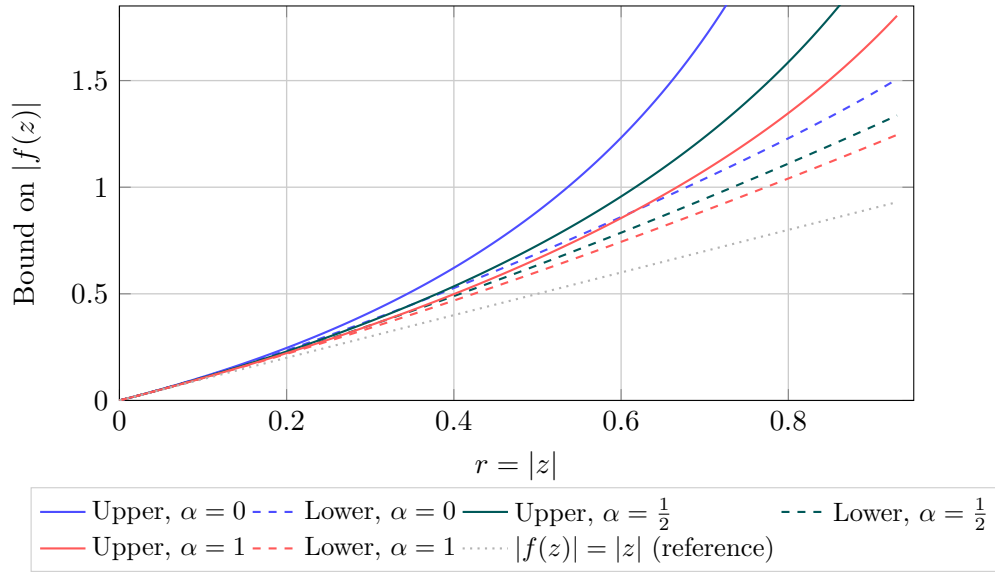
\begin{figure}[ht]
	\centering
	\pgfplotsset{
		growth/.style={
			width=0.78\textwidth, height=6.8cm,
			xlabel={$r = |z|$},
			ylabel={Bound on $|f(z)|$},
			xmin=0, xmax=0.95,
			ymin=0, ymax=1.85,
			xtick={0,0.2,0.4,0.6,0.8},
			ytick={0,0.5,1.0,1.5},
			grid=both,
			grid style={line width=0.3pt, draw=gray!25},
			major grid style={line width=0.5pt, draw=gray!40},
			tick label style={font=\small},
			label style={font=\small},
			legend style={
				at={(0.5,-0.22)},
				anchor=north,
				legend columns=4,
				font=\footnotesize,
				draw=gray!40,
				fill=white
			},
			legend cell align=left,
			every axis plot/.append style={line width=1.1pt},
		}
	}
	\begin{tikzpicture}
		\begin{axis}[growth]

			\addplot[color=blue!70, thick]
			expression[domain=0:0.93, samples=120]
			{ x
				+ 2*x^2/2  + 2*x^3/3  + 2*x^4/4  + 2*x^5/5
				+ 2*x^6/6  + 2*x^7/7  + 2*x^8/8  + 2*x^9/9
				+ 2*x^10/10+ 2*x^11/11+ 2*x^12/12
			};
			\addlegendentry{Upper, $\alpha=0$}

			\addplot[color=blue!70, thick, dashed]
			expression[domain=0:0.93, samples=120]
			{ x
				+ 2*x^2/2  - 2*x^3/3  + 2*x^4/4  - 2*x^5/5
				+ 2*x^6/6  - 2*x^7/7  + 2*x^8/8  - 2*x^9/9
				+ 2*x^10/10- 2*x^11/11+ 2*x^12/12
			};
			\addlegendentry{Lower, $\alpha=0$}

			\addplot[color=teal!70!black, thick]
			expression[domain=0:0.93, samples=120]
			{ x
				+ 2*x^2/(0.5*4+0.5*2)
				+ 2*x^3/(0.5*9+0.5*3)
				+ 2*x^4/(0.5*16+0.5*4)
				+ 2*x^5/(0.5*25+0.5*5)
				+ 2*x^6/(0.5*36+0.5*6)
				+ 2*x^7/(0.5*49+0.5*7)
				+ 2*x^8/(0.5*64+0.5*8)
				+ 2*x^9/(0.5*81+0.5*9)
				+ 2*x^10/(0.5*100+0.5*10)
				+ 2*x^11/(0.5*121+0.5*11)
				+ 2*x^12/(0.5*144+0.5*12)
			};
			\addlegendentry{Upper, $\alpha=\tfrac{1}{2}$}
			
			\addplot[color=teal!70!black, thick, dashed]
			expression[domain=0:0.93, samples=120]
			{ x
				+ 2*x^2/(0.5*4+0.5*2)
				- 2*x^3/(0.5*9+0.5*3)
				+ 2*x^4/(0.5*16+0.5*4)
				- 2*x^5/(0.5*25+0.5*5)
				+ 2*x^6/(0.5*36+0.5*6)
				- 2*x^7/(0.5*49+0.5*7)
				+ 2*x^8/(0.5*64+0.5*8)
				- 2*x^9/(0.5*81+0.5*9)
				+ 2*x^10/(0.5*100+0.5*10)
				- 2*x^11/(0.5*121+0.5*11)
				+ 2*x^12/(0.5*144+0.5*12)
			};
			\addlegendentry{Lower, $\alpha=\tfrac{1}{2}$}

			\addplot[color=red!65, thick]
			expression[domain=0:0.93, samples=120]
			{ x
				+ 2*x^2/4   + 2*x^3/9   + 2*x^4/16
				+ 2*x^5/25  + 2*x^6/36  + 2*x^7/49
				+ 2*x^8/64  + 2*x^9/81  + 2*x^10/100
				+ 2*x^11/121+ 2*x^12/144
			};
			\addlegendentry{Upper, $\alpha=1$}

			\addplot[color=red!65, thick, dashed]
			expression[domain=0:0.93, samples=120]
			{ x
				+ 2*x^2/4   - 2*x^3/9   + 2*x^4/16
				- 2*x^5/25  + 2*x^6/36  - 2*x^7/49
				+ 2*x^8/64  - 2*x^9/81  + 2*x^10/100
				- 2*x^11/121+ 2*x^12/144
			};
			\addlegendentry{Lower, $\alpha=1$}
			
			\addplot[color=black!30, thick, dotted]
			expression[domain=0:0.93]{x};
			\addlegendentry{$|f(z)|=|z|$ (reference)}
			
		\end{axis}
	\end{tikzpicture}
	\caption{Upper bounds (solid) and lower bounds (dashed) for
		$|f(z)|$ from Theorem~D, plotted as functions of $r = |z|$ for
		$\alpha\in\{0,\tfrac{1}{2},1\}$.
		The partial sums are truncated at $m = 12$. The dotted line
		$|f(z)| = |z|$ is shown for reference. 
		}
	\label{fig:growth-bounds}
\end{figure}

The result below provides a condition that is sufficient for a function to belong to the class $f\in\mathcal{W}_{\mathcal{H}}^0(\alpha)$.
\begin{theoE}\emph{\cite[Theorem 4.5]{Ghosh-Allu-2019}}
Let $f=h+\ol g\in \mathcal{S}_{\mathcal{H}}^0$ be given by \eqref{1.2}. If
\begin{align*} 
\sum\limits_{m=2}^{\infty}\left(\alpha m^2+\left(1-\alpha\right)m\right)\left(|a_m|+|b_m|\right)<1,
\end{align*}
	then $f\in\mathcal{W}_{\mathcal{H}}^0(\alpha)$.
\end{theoE}
In the next result Ghosh and Vasudevarao \cite{Ghosh-Allu-2019} provide $\mathcal{W}_{\mathcal{H}}^0(\alpha)$ is closed under convex combinations. 
\begin{theoF}\emph{\cite[Theorem 5.1]{Ghosh-Allu-2019}}
	$\mathcal{W}_{\mathcal{H}}^0(\alpha)$ is closed under convex combinations.
\end{theoF}

\begin{figure}[t]
	\centering
	
	\begin{tikzpicture}[scale=0.70,
		transform shape,
		>=Stealth,
		every node/.style={font=\small},
		topic/.style={
			draw,
			rounded corners=3pt,
			minimum width=3.3cm,
			minimum height=0.9cm,
			align=center,
			fill=blue!6
		},
		subtopic/.style={
			draw,
			rounded corners=3pt,
			minimum width=3.0cm,
			minimum height=0.85cm,
			align=center,
			fill=green!8
		},
		paper/.style={
			draw,
			rounded corners=3pt,
			minimum width=4.2cm,
			minimum height=1cm,
			align=center,
			fill=red!15,
			very thick
		}
		]
		
		
		\node[topic] (SCV)
		{Several Complex Variables};
		
		
		\node[subtopic]
		(Holo)
		[below left=1.8cm and 2.7cm of SCV]
		{Holomorphic\\Mappings};
		
		\node[subtopic]
		(Harm)
		[below=1.8cm of SCV]
		{Harmonic\\Mappings};
		
		\node[subtopic]
		(Pluri)
		[below right=1.8cm and 2.7cm of SCV]
		{Pluriharmonic\\Mappings};
		
		
		\node[subtopic]
		(Uni)
		[below=2.0cm of Holo]
		{Univalence};
		
		\node[subtopic]
		(Dist)
		[below=2.0cm of Harm]
		{Coefficient Estimates\\Distortion};
		
		\node[subtopic]
		(Geo)
		[below=2.0cm of Pluri]
		{Starlike\\Convex\\Close-to-Convex};
		
		
		\node[paper]
		(Paper)
		[below=2.3cm of Dist]
		{Present Investigation};
		
		
		\draw[->,thick]
		(SCV)--(Holo);
		
		\draw[->,thick]
		(SCV)--(Harm);
		
		\draw[->,thick]
		(SCV)--(Pluri);
		
		\draw[->,thick]
		(Holo)--(Uni);
		
		\draw[->,thick]
		(Harm)--(Dist);
		
		\draw[->,thick]
		(Pluri)--(Geo);
		
		\draw[->,thick]
		(Uni)--(Paper);
		
		\draw[->,thick]
		(Dist)--(Paper);
		
		\draw[->,thick]
		(Geo)--(Paper);
		
		
		\begin{scope}[on background layer]
			
			\node[
			draw=red!70,
			rounded corners=5pt,
			line width=1pt,
			inner sep=6pt,
			fit=(Paper),
			fill=red!5
			]{};
			
		\end{scope}
		
	\end{tikzpicture}
	
	\caption{
		Research landscape illustrating the mathematical framework of the present investigation.
	}
	
	\label{FigResearchLandscape}
	
\end{figure}
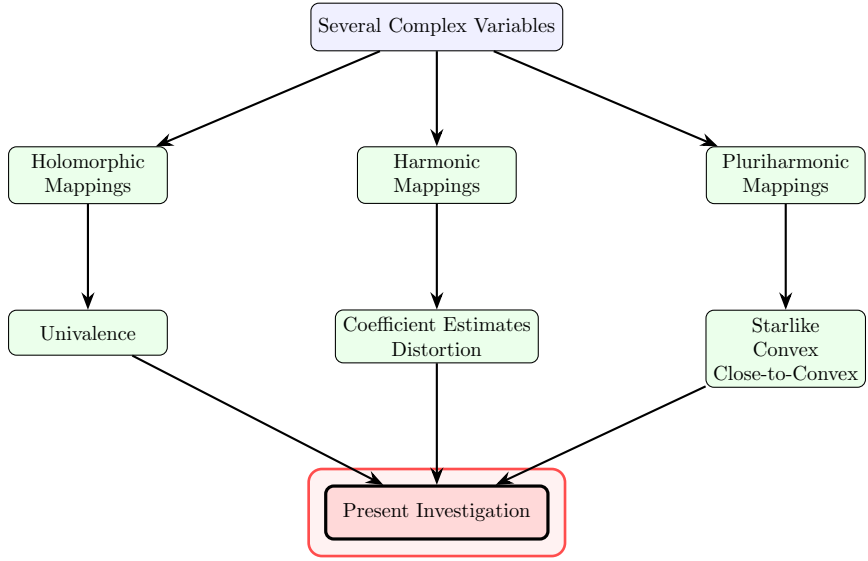
\begin{tcolorbox}[
	enhanced,
	colback=blue!4,
	colframe=blue!60!black,
	boxrule=0.8pt,
	arc=2mm,
	left=3mm,
	right=3mm,
	top=2mm,
	bottom=2mm,
	title=\textbf{Motivation},
	coltitle=black,
	fonttitle=\bfseries,
	colbacktitle=blue!15,
	attach boxed title to top left=
	{yshift=-2mm,xshift=3mm},
	boxed title style={
		boxrule=0.6pt,
		colframe=blue!60!black,
		colback=blue!15,
		arc=1.5mm
	}
	]
	The systematic study of geometric properties of harmonic mappings in the plane
	dates back to Clunie and Sheil--Small \cite{Clunie-Sheil-Small-1984}.
	Subsequent extensions to pluriharmonic mappings and higher dimensions include
	Aizenberg et al. \cite{Aizenberg-Aytuna-Djakov-JMAA-2001},
	Li--Ponnusamy \cite{Li-Ponnusamy-2013}, and recent work on coefficient
	problems and radius theorems
	\cite{Nagpal-Ravichandran-2014,Ghosh-Allu-2019}. Our results complement and
	extend these works by providing sharp coefficient bounds in several complex
	variables and explicit radii for geometric properties under natural
	normalizations.
\end{tcolorbox}
\section{{\bf Main Results}}\label{Sec-2}
We now introduce the class $\mathcal{W}_{\mathcal{H}_n^0}(\alpha)$ of pluriharmonic mappings defined on $\mathbb{P} \Delta(0;1)$ as follows.
\begin{defi} For $\alpha\geq 0$ and $z\in \mathbb{P} \Delta(0;1)$, let
	\begin{align*}
		&\mathcal{W}_{\mathcal{H}_n^0}(\alpha)\\=&\left\lbrace h+\ol g\in \mathcal{H}_n^0:\Re\left(\frac{\partial h(z)}{\partial z_k}+\alpha z_l\frac{\partial^2 h(z)}{\partial z_j\partial z_k}\right)>\left|\frac{\partial g(z)}{\partial z_k}+\alpha z_l\frac{\partial^2 g(z)}{\partial z_j\partial z_k}\right|\;\forall j,k,l\in\mathbb{Z}[1,n] \right\rbrace.
	\end{align*}
\end{defi}

When $n=1$, we denote $\mathcal{W}_{\mathcal{H}_1^0}(\alpha)$ by $\mathcal{W}_{\mathcal{H}}^0(\alpha)$.

\smallskip
We will show that the class $\mathcal{W}_{\mathcal{H}_n^0}(\alpha)$ is closely related to the class
\begin{align*}
	\mathcal{W}_n(\alpha)=\left\lbrace \phi\in\mathcal{A}_n: \Re\left(\frac{\partial \phi(z)}{\partial z_k}+\alpha z_l\frac{\partial^2 \phi(z)}{\partial z_j\partial z_k}\right)>0\;\forall\;z\in \mathbb{P} \Delta(0;1)\; \text{ and }\; j,k,l\in\mathbb{Z}[1,n]\right\rbrace.
\end{align*}

When $n=1$, we denote $\mathcal{W}_1(\alpha)$ by $\mathcal{W}(\alpha)$.  \vspace{1.2mm}

The novelty of this research lies in the extension of the theory of \emph{harmonic} functions in $\mathbb{C}$ to \emph{pluriharmonic} functions in higher-dimensional complex spaces ($\mathbb{C}^n$). In this paper, we study the analytic and geometric properties of the class $\mathcal{W}_{\mathcal{H}_n^0}(\alpha)$. We also establish that the subclass $\mathcal{W}_{\mathcal{H}_n^0}(\alpha)$ is a natural pluriharmonic extension of the holomorphic class $\mathcal W_n(\alpha)$.
The sharp coefficient bounds and growth estimates provided for the class $\mathcal{W}_{\mathcal{H}_n^0}(\alpha)$ constitute new contributions to the study of geometric function theory in several complex variables.\vspace{1.2mm}

The following result provides a one-to-one correspondence between the classes $\mathcal{W}_{\mathcal{H}_n^0}(\alpha)$ and $\mathcal{W}_n(\alpha)$.

\begin{theo}\label{Th-1.1} A pluriharmonic mapping $f=h+\ol g$ is in $\mathcal{W}_{\mathcal{H}_n^0}(\alpha)$ if and only if the function  $F_{\varepsilon}=h+\varepsilon g$ is in $\mathcal{W}_n(\alpha)$ for each $\varepsilon\; (|\varepsilon|=1)$.
\end{theo}

\begin{proof} Let $f\in \mathcal{W}_{\mathcal{H}_n^0}(\alpha)$. Then for every $\varepsilon$ satisfying $|\varepsilon|=1$, we have
\begin{align*}
\Re\left(\frac{\partial F_{\varepsilon}(z)}{\partial z_k}+\alpha z_l\frac{\partial^2 F_{\varepsilon}(z)}{\partial z_j\partial z_k}\right)
=&\Re\left(\frac{\partial h(z)}{\partial z_k}+\alpha z_l\frac{\partial^2 h(z)}{\partial z_j\partial z_k}\right)+\Re\left(\varepsilon\left(\frac{\partial g(z)}{\partial z_k}+\alpha z_l\frac{\partial^2 g(z)}{\partial z_j\partial z_k}\right)\right)\\>&
\Re\left(\frac{\partial h(z)}{\partial z_k}+\alpha z_l\frac{\partial^2 h(z)}{\partial z_j\partial z_k}\right)-\left|\frac{\partial g(z)}{\partial z_k}+\alpha z_l\frac{\partial^2 g(z)}{\partial z_j\partial z_k}\right|\\>&0
\end{align*}
for all $z\in \mathbb{P} \Delta(0_n;1_n)$ and for all $j,k,l\in\mathbb{Z}[1,n]$ and so $F_{\varepsilon}\in \mathcal{W}_n(\alpha)$. Conversely, let $F_{\varepsilon}=h+\varepsilon g\in \mathcal{W}_n(\alpha)$ for each $\varepsilon\;(|\varepsilon|=1)$. Then, we have
\begin{align*}
\Re\left(\frac{\partial F_{\varepsilon}(z)}{\partial z_k}+\alpha z_l\frac{\partial^2 F_{\varepsilon}(z)}{\partial z_j\partial z_k}\right)
=\Re\left(\frac{\partial h(z)}{\partial z_k}+\alpha z_l\frac{\partial^2 h(z)}{\partial z_j\partial z_k}\right)+\varepsilon\left(\frac{\partial g(z)}{\partial z_k}+\alpha z_l\frac{\partial^2 g(z)}{\partial z_j\partial z_k}\right)>0
\end{align*}
for all $z\in \mathbb{P} \Delta(0_n;1_n)$ and for all $j,k,l\in\mathbb{Z}[1,n]$. Replacing $\varepsilon$ by $-\varepsilon$ in the above inequality, we have
\begin{align}\label{Th1:1.1}
\Re\left(\frac{\partial h(z)}{\partial z_k}+\alpha z_l\frac{\partial^2 h(z)}{\partial z_j\partial z_k}\right)>\Re\left(\varepsilon\left(\frac{\partial g(z)}{\partial z_k}+\alpha z_l\frac{\partial^2 g(z)}{\partial z_j\partial z_k}\right)\right)
\end{align}
for all $z\in \mathbb{P} \Delta(0_n;1_n)$ and for all $j,k,l\in\mathbb{Z}[1,n]$. Since $\varepsilon$ is arbitrary, we may choose it so that 
\begin{align}\label{Th1:1.2}
\Re\left(\varepsilon\left(\frac{\partial g(z)}{\partial z_k}+\alpha z_l\frac{\partial^2 g(z)}{\partial z_j\partial z_k}\right)\right)=\left|\varepsilon\left(\frac{\partial g(z)}{\partial z_k}+\alpha z_l\frac{\partial^2 g(z)}{\partial z_j\partial z_k}\right)\right|
\end{align}
for all $z\in \mathbb{P} \Delta(0_n;1_n)$ and for all $j,k,l\in\{1,2,\ldots,n\}$.
For the verification of the equality (\ref{Th1:1.2}), we substitute
\begin{align*}
\varepsilon=\frac{\ol{\frac{\partial g(z)}{\partial z_k}+\alpha z_l\frac{\partial^2 g(z)}{\partial z_j\partial z_k}}}{\left|\frac{\partial g(z)}{\partial z_k}+\alpha z_l\frac{\partial^2 g(z)}{\partial z_j\partial z_k}\right|}
\end{align*}
when $\frac{\partial g(z)}{\partial z_k}+\alpha z_l\frac{\partial^2 g(z)}{\partial z_j\partial z_k}\neq 0$. Obviously (\ref{Th1:1.2}) is true when $\frac{\partial g(z)}{\partial z_k}+\alpha z_l\frac{\partial^2 g(z)}{\partial z_j\partial z_k}= 0$. Now using (\ref{Th1:1.2}) to (\ref{Th1:1.1}), we get
\begin{align*}
\Re\left(\frac{\partial h(z)}{\partial z_k}+\alpha z_l\frac{\partial^2 h(z)}{\partial z_j\partial z_k}\right)>\Re\left|\frac{\partial g(z)}{\partial z_k}+\alpha z_l\frac{\partial^2 g(z)}{\partial z_j\partial z_k}\right|
\end{align*}
for all $z\in \mathbb{P} \Delta(0_n;1_n)$ and for all $j,k,l\in\{1,2,\ldots,n\}$. Consequently $f\in \mathcal{W}_{\mathcal{H}_n^0}(\alpha)$.
\end{proof}

The following result provides the sharp coefficient bounds for functions in $\mathcal{P}_{\mathcal{H}_n^0}(M)$.

\begin{theo}\label{Th-1.2} Let $f=h+\ol g\in \mathcal{W}_{\mathcal{H}_n^0}(\alpha)$ and be given by (\ref{Eq 1.7}). Then for any multi-index $\beta=(\beta_1,\beta_2,\ldots,\beta_n)$ such that $|\beta|=m\geq 2$, we have
\begin{align*}
\sum\limits_{|\beta|=m}|b_{\beta}|\leq \frac{\binom{m+n-1}{n-1}n^2}{\alpha m^2+(n-\alpha)m}
\end{align*}
The inequality is sharp.
\end{theo}

\begin{proof} 
\begin{table}[H]
	\centering
	\caption{Notation used throughout the proof.}
	
	\begin{tcolorbox}[
		colback=white,
		colframe=black,
		boxrule=0.8pt,
		arc=2mm,
		width=0.78\textwidth,
		left=3mm,
		right=3mm,
		top=2mm,
		bottom=2mm]
		\centering
		\renewcommand{\arraystretch}{1.15}
		\begin{tabular}{ll}
			\hline
			\textbf{Symbol} & \textbf{Meaning}\\
			\hline
			$\beta$ & Multi-index of degree $m$\\
			$\beta_k$ & Exponent of $z_k$\\
			$\beta_{jk}$ & Mixed derivative coefficient\\
			$\beta_j^*$ & Reduced exponent after differentiation\\
			$\beta_k^{**}$ & Modified exponent after differentiation\\
			$P_m,Q_m$ & Homogeneous polynomials of degree $m$\\
			\hline
		\end{tabular}
	\end{tcolorbox}
	
\end{table}	
	
	Let $f=h+\ol g\in \mathcal{W}_{\mathcal{H}_n^0}(\alpha)$. Then
\begin{align}\label{Th2:1.1}
\Re\left(\frac{\partial h(z)}{\partial z_k}+\alpha z_l\frac{\partial^2 h(z)}{\partial z_j\partial z_k}\right)>\left|\frac{\partial g(z)}{\partial z_k}+\alpha z_l\frac{\partial^2 g(z)}{\partial z_j\partial z_k}\right|
\end{align}
for all $z=(z_1,z_2,\ldots,z_n)\in \mathbb{P} \Delta(0_n;1_n)$ and for all $j,k,l\in\mathbb{Z}[1,n]$. It follows from (\ref{Eq 1.7}) that
\begin{align}\label{Th2:1.2}
h(z)=\sum\limits_{j=1}^n z_j+\sum\limits_{m=2}^{\infty}P_m(z)\quad {and} \quad g(z)=\sum\limits_{k=2}^{\infty} Q_m(z)
\end{align}
for all $z\in \mathbb{P} \Delta(0_n;1_n)$, where
\begin{align}\label{Th2:1.3}
P_{m}(z)=\sum\limits_{|\beta|=m} a_{\beta} z^{\beta}\quad \text{and}\quad Q_m(z)=\sum\limits_{|\beta|=m} b_{\beta} z^{\beta}
\end{align}
are homogeneous polynomials of degree $m\geq 2$ in $z\in \mathbb{P} \Delta(0_n;1_n)$. We know that the number of terms in $\sum\limits_{|\beta|=m}$ is $\binom{|\beta|+n-1}{n-1}$. A simple computation using (\ref{Th2:1.3}) shows that
\begin{align}\label{Th2:1.3a}
\frac{\partial Q_m(z)}{\partial z_k}=\sum\limits_{|\beta|=m}\beta_{k} b_{\beta}z_1^{\beta_1}\ldots z_{j-1}^{\beta_{j-1}}z_{j}^{\beta_j}z_{j+1}^{\beta_{j+1}}\ldots z_{k-1}^{\beta_{k-1}}z_{k}^{\beta_k-1}z_{k+1}^{\beta_{k+1}}\ldots z_l^{\beta_l}\ldots z_n^{\beta_n}
\end{align}
and
\begin{align}\label{Th2:1.4}
z_l\frac{\partial^2 Q_m(z)}{\partial z_j\partial z_k}=\sum\limits_{|\beta|=m}\beta_{jk} b_{\beta}z_1^{\beta_1}\ldots z_{j-1}^{\beta_{j-1}}z_{j}^{\beta_j^*}z_{j+1}^{\beta_{j+1}}\ldots z_{k-1}^{\beta_{k-1}}z_{k}^{\beta_k^{**}}z_{k+1}^{\beta_{k+1}}\ldots z_l^{\beta_l+1}\ldots z_n^{\beta_n},
\end{align}
where
\begin{align}\label{ll.1}
\beta_{jk}=
\begin{cases}
\beta_j(\beta_j-1), & \text{if}\; j=k\;\text{and}\;\beta_j=|\beta|,\\[2ex]
\beta_j(\beta_j-1), & \text{if}\; j=k\;\text{and}\;2\leq \beta_j<|\beta|,\\[2ex]
\beta_j\beta_k,& \text{if}\; j\neq k
\end{cases}
\;\;,\quad 
\beta^*_j=
\begin{cases}
\beta_j-2,& \text{if}\; j=k,\\[2ex]
\beta_j-1,& \text{if}\; j\neq k
\end{cases}
\end{align}
and
\begin{align}\label{ll.2}
\beta^{**}_k=
\begin{cases}
\beta_k,& \text{if}\; j=k,\\[2ex]
\beta_k-1,& \text{if}\; j\neq k.
\end{cases}
\end{align}

Using (\ref{Th2:1.2})-(\ref{Th2:1.3a}) and (\ref{Th2:1.4}), we get respectively
\begin{align}\label{Th2:1.5a}
\frac{\partial g(z)}{\partial z_k}&=
\sum\limits_{m=2}^{\infty} \frac{\partial Q_m(z)}{\partial z_k}\\&=\sum\limits_{m=2}^{\infty}\sum\limits_{|\beta|=m}\beta_{k} b_{\beta}z_1^{\beta_1}\ldots z_{j-1}^{\beta_{j-1}}z_{j}^{\beta_j}z_{j+1}^{\beta_{j+1}}\ldots z_{k-1}^{\beta_{k-1}}z_{k}^{\beta_k-1}z_{k+1}^{\beta_{k+1}}\ldots z_l^{\beta_l}\ldots z_n^{\beta_n}\nonumber
\end{align}
and 
\begin{align}\label{Th2:1.5}
z_l\frac{\partial^2 g(z)}{\partial z_j\partial z_k}&=
\sum\limits_{m=2}^{\infty} z_l\frac{\partial^2 Q_m(z)}{\partial z_j\partial z_k}\\&=\sum\limits_{m=2}^{\infty}\sum\limits_{|\beta|=m}\beta_{jk} b_{\beta}z_1^{\beta_1}\ldots z_{j-1}^{\beta_{j-1}}z_{j}^{\beta_j^*}z_{j+1}^{\beta_{j+1}}\ldots z_{k-1}^{\beta_{k-1}}z_{k}^{\beta_k^{**}}z_{k+1}^{\beta_{k+1}}\ldots z_l^{\beta_l+1}\ldots z_n^{\beta_n}.\nonumber
\end{align}

By applying Cauchy's integral formula to $\frac{\partial h(z)}{\partial z_k}$ and $z_l \frac{\partial^2 h(z)}{\partial z_j \partial z_k}$, we find from \eqref{Th2:1.5a} and \eqref{Th2:1.5} that
\begin{align}\label{Th2:1.6a}
&(2\pi i)^n \beta_{k}b_{\beta}\\=&\int\limits_{|z_1|=r_1}\ldots \int\limits_{|z_n|=r_n}\frac{\frac{\partial g(z)}{\partial z_k}\;dz_1\;d z_2\ldots d z_n}{z_1^{\beta_1+1}\ldots z_{j-1}^{\beta_{j-1}+1}z_{j}^{\beta_j+1}z_{j+1}^{\beta_{j+1}+1}\ldots z_{k-1}^{\beta_{k-1}+1}z_{k}^{\beta_k}z_{k+1}^{\beta_{k+1}+1}\ldots z_l^{\beta_l+1}\ldots z_n^{\beta_n+1}},\nonumber
\end{align}
and
\begin{align}\label{Th2:1.6}
&(2\pi i)^n \beta_{jk}b_{\beta}\\=&\int\limits_{|z_1|=r_1}\ldots \int\limits_{|z_n|=r_n}\frac{z_l \frac{\partial^2 g(z)}{\partial z_j\partial z_k}\;dz_1\;d z_2\ldots d z_n}{z_1^{\beta_1+1}\ldots z_{j-1}^{\beta_{j-1}+1}z_{j}^{\beta_j^*+1}z_{j+1}^{\beta_{j+1}+1}\ldots z_{k-1}^{\beta_{k-1}+1}z_{k}^{\beta_k^{**}+1}z_{k+1}^{\beta_{k+1}+1}\ldots z_l^{\beta_l+2}\ldots z_n^{\beta_n+1}},\nonumber
\end{align}
where $0<r_j<1$ for $j=1,2,\ldots,n$. We set $z=\left(re^{\iota\theta},re^{\iota\theta},\ldots,re^{\iota\theta}\right)$, where $0\leq \theta\leq 2\pi$.
Now from (\ref{Th2:1.6a}) and (\ref{Th2:1.6}), we have respectively
\begin{align}\label{Th2:1.6b}
(2\pi i)^n \beta_{k}b_{\beta}=\int\limits_{0}^{2\pi}\ldots \int\limits_0^{2\pi}\frac{\frac{\partial g(z)}{\partial z_k}\;(d\theta)^n}{r^{m-1}e^{\iota (m-1)\theta}}
\end{align}
and
\begin{align}\label{Th2:1.6c}
(2\pi i)^n \beta_{jk}b_{\beta}=\int\limits_{0}^{2\pi}\ldots \int\limits_0^{2\pi}\frac{re^{\iota \theta} \frac{\partial^2 g(z)}{\partial z_j\partial z_k}\;(d\theta)^n}{r^{m-1}e^{\iota (m-1)\theta}}.
\end{align}
\noindent Clearly from (\ref{Th2:1.6b}) and (\ref{Th2:1.6c}), we get
\begin{align*}
(2\pi i)^n \left(\beta_k+\alpha \beta_{jk}\right)b_{\beta}=\int\limits_{0}^{2\pi}\ldots \int\limits_0^{2\pi}\frac{\left(\frac{\partial g(z)}{\partial z_k}+\alpha re^{\iota \theta} \frac{\partial^2 g(z)}{\partial z_j\partial z_k}\right)\;(d\theta)^n}{r^{m-1}e^{\iota (m-1)\theta}}
\end{align*}
and so
\begin{align}\label{Th2:1.7}
(2\pi)^n\sum\limits_{|\beta|=m}\left(\beta_k+\alpha \beta_{jk}\right)|b_{\beta}|\leq \binom{|\beta|+n-1}{n-1}\int\limits_{0}^{2\pi}\ldots \int\limits_0^{2\pi}\frac{\left|\frac{\partial g(z)}{\partial z_k}+\alpha re^{\iota \theta} \frac{\partial^2 g(z)}{\partial z_j\partial z_k}\right|\;(d\theta)^n}{r^{m-1}}
\end{align}
for all $j,k,l\in\mathbb{Z}[1,n]$. 

Consequently, from (\ref{Th2:1.7}), we obtain
\begin{align}\label{Th2:1.8}
&(2\pi)^nr^{m-1} \sum\limits_{l=1}^n\sum\limits_{|\beta|=m}\sum\limits_{j=1}^n\sum\limits_{k=1}^n \left(\beta_k+\alpha \beta_{jk}\right)|b_{\beta}|\\=&
(2\pi)^nr^{m-1} \sum\limits_{|\beta|=m}\sum\limits_{l=1}^n\sum\limits_{j=1}^n\sum\limits_{k=1}^n \left(\beta_k+\alpha \beta_{jk}\right)|b_{\beta}|\nonumber\\ \leq &\binom{|\beta|+n-1}{n-1}\int\limits_{0}^{2\pi}\ldots \int\limits_0^{2\pi}\sum\limits_{l=1}^n\sum\limits_{j=1}^n\sum\limits_{k=1}^n\left|\frac{\partial g(z)}{\partial z_k}+\alpha re^{\iota \theta} \frac{\partial^2 g(z)}{\partial z_j\partial z_k}\right|\;(d\theta)^n.\nonumber
\end{align}
We now observe that if $\beta_j=|\beta|=m$ for $j=1,2,\dots,n$, then
\begin{align*}
\sum\limits_{|\beta|=m}\sum\limits_{j=1}^n\sum\limits_{k=1}^n\beta_{jk}|b_{\beta}|=\sum\limits_{|\beta|=m}\sum\limits_{j=1}^n \beta_j(\beta_j-1)|b_{\beta}|=m(m-1)\sum\limits_{\substack{j=1\\\beta_j=m}}^n|b_{\beta}|
\end{align*}
and if $\beta_j<|\beta|$ for $j=1,2,\ldots,n$, then we have
\begin{align*}
\sum\limits_{|\beta|=m}\sum\limits_{j=1}^n\sum\limits_{k=1}^n\beta_{jk}|b_{\beta}|=&
\sum\limits_{|\beta|=m}\left(\sum\limits_{j=1}^n \beta_j(\beta_j-1)+2\sum\limits_{\substack{j,k=1\\j\neq k}}^n \beta_j\beta_k\right)|b_{\beta}|\\=&m(m-1)\sum\limits_{\substack{|(\beta_1,\beta_2,\ldots,\beta_n)|=m\\\beta_j<m}}|b_{\beta}|.
\end{align*}
Therefore,
\begin{align}\label{Th2:1.8a}
\sum\limits_{|\beta|=m}\sum\limits_{j=1}^n\sum\limits_{k=1}^n\beta_{jk}|b_{\beta}|=m(m-1)\sum\limits_{|\beta|=m}|b_{\beta}|.
\end{align}

Using (\ref{Th2:1.8a}) in (\ref{Th2:1.8}), we get
\begin{align}\label{Th2:1.8b}
&(2\pi)^nr^{m-1}n(nm+\alpha m(m-1)) \sum\limits_{|\beta|=m}|b_{\beta}|
\\ \leq &\binom{|\beta|+n-1}{n-1}\int\limits_{0}^{2\pi}\ldots \int\limits_0^{2\pi}\sum\limits_{l=1}^n\sum\limits_{j=1}^n\sum\limits_{k=1}^n\left|\frac{\partial g(z)}{\partial z_k}+\alpha re^{\iota \theta} \frac{\partial^2 g(z)}{\partial z_j\partial z_k}\right|\;(d\theta)^n.\nonumber
\end{align}

For any multi-index $\nu=(\nu_1,\nu_2,\ldots,\nu_n)$, we have
\begin{align}\label{Th2:1.9}
\int\limits_{0}^{2\pi}\ldots \int\limits_0^{2\pi} (e^{i\theta_1})^{\nu_1}\ldots (e^{i\theta_n})^{\nu_n}d\theta_1 \ldots d\theta_n=
\begin{cases}
0,& \nu\neq (0,0,\ldots,0),\\[2ex]
(2\pi)^n,& \nu=(0,0,\ldots,0).
\end{cases}
\end{align}
Moreover, 
\begin{align}\label{Th2:1.8c}
&\frac{\partial h(z)}{\partial z_k}+\alpha z_l\frac{\partial^2 h(z)}{\partial z_j\partial z_k}\\=&1+
\sum\limits_{m=2}^{\infty}\sum\limits_{|\beta|=m}\beta_{k} a_{\beta}z_1^{\beta_1}\ldots z_{j-1}^{\beta_{j-1}}z_{j}^{\beta_j}z_{j+1}^{\beta_{j+1}}\ldots z_{k-1}^{\beta_{k-1}}z_{k}^{\beta_k-1}z_{k+1}^{\beta_{k+1}}\ldots z_l^{\beta_l}\ldots z_n^{\beta_n}\nonumber\\&+\alpha 
\sum\limits_{m=2}^{\infty}\sum\limits_{|\beta|=m}\beta_{jk} a_{\beta}z_1^{\beta_1}\ldots z_{j-1}^{\beta_{j-1}}z_{j}^{\beta_j^*}z_{j+1}^{\beta_{j+1}}\ldots z_{k-1}^{\beta_{k-1}}z_{k}^{\beta_k^{**}}z_{k+1}^{\beta_{k+1}}\ldots z_l^{\beta_l+1}\ldots z_n^{\beta_n}.\nonumber
\end{align}

Using (\ref{Th2:1.9}) to (\ref{Th2:1.8c}), we deduce that
\begin{align*}
\frac{1}{(2\pi)^n}\int\limits_{0}^{2\pi}\ldots \int\limits_0^{2\pi} \left(\frac{\partial h(z)}{\partial z_k}+\alpha z_l\frac{\partial^2 h(z)}{\partial z_j\partial z_k}\right)\;d\theta_1\;d \theta_2\ldots d \theta_n=1
\end{align*}
and so for $z=\left(re^{\iota\theta},re^{\iota\theta},\ldots,re^{\iota\theta}\right)$, where $0\leq \theta\leq 2\pi$, we get
\begin{align}\label{Th2:1.9a}
\frac{1}{(2\pi)^n}\int\limits_{0}^{2\pi}\ldots \int\limits_0^{2\pi} \Re\left(\frac{\partial h(z)}{\partial z_k}+\alpha re^{\iota \theta}\frac{\partial^2 h(z)}{\partial z_j\partial z_k}\right)\;(d\theta)^n=1
\end{align}
for all $j,k,l\in\mathbb{Z}[1,n]$.
In view of (\ref{Th2:1.1}) and (\ref{Th2:1.9a}) and using (\ref{Th2:1.8b}), we obtain
\begin{align*}
&r^{m-1}n(nm+\alpha m(m-1)) \sum\limits_{|\beta|=m}|b_{\beta}|
\\ \leq &\binom{|\beta|+n-1}{n-1}\frac{1}{(2\pi)^n}\int\limits_{0}^{2\pi}\ldots \int\limits_0^{2\pi}\sum\limits_{l=1}^n\sum\limits_{j=1}^n\sum\limits_{k=1}^n\left|\frac{\partial g(z)}{\partial z_k}+\alpha re^{\iota \theta} \frac{\partial^2 g(z)}{\partial z_j\partial z_k}\right|\;(d\theta)^n\\\leq &
\binom{|\beta|+n-1}{n-1}\frac{1}{(2\pi)^n}\int\limits_{0}^{2\pi}\ldots \int\limits_0^{2\pi}\sum\limits_{l=1}^n\sum\limits_{j=1}^n\sum\limits_{k=1}^n\Re\left(\frac{\partial h(z)}{\partial z_k}+\alpha re^{\iota \theta}\frac{\partial^2 h(z)}{\partial z_j\partial z_k}\right)\;(d\theta)^n\\\leq&
n^3\binom{|\beta|+n-1}{n-1}.
\end{align*}

Letting $r\to 1^-$, if $m\geq 2$, we obtain
\begin{align*}
\sum\limits_{|\beta|=m}|b_{\beta}|\leq \frac{\binom{m+n-1}{n-1}n^2}{\alpha m^2+(n-\alpha)m}.
\end{align*}
\vspace{1.2mm}

To show that the bound is sharp, we consider the following function
\begin{align*}
f_1(z)=\sum\limits_{j=1}^n z_j+\sum\limits_{m=2}^{\infty}\sum\limits_{|\beta|=m} b_{\beta} \ol{z^{\beta}},
\end{align*}
where 
\begin{align*}
b_{\beta}=\frac{n^2}{\alpha m^2+(n-\alpha)m}
\end{align*}
for all multi-index $\beta=(\beta_1,\beta_2,\ldots,\beta_n)$ such that $|\beta|=m$. 
It is easy to see that $f_1\in \mathcal{W}_{\mathcal{H}_n^0}(\alpha)$, and 
\begin{align*}
\sum\limits_{|\beta|=m}|b_{\beta}(f_1)|=\frac{\binom{m+n-1}{n-1}n^2}{\alpha m^2+(n-\alpha)m}.
\end{align*} 
\begin{table}[h]
	\centering
	\caption{Main ingredients in the proof of Theorem \ref{Th-1.2}.}
	\begin{tabular}{ll}
		\hline
		Step & Tool Used\\
		\hline
		Series expansion & Homogeneous polynomials\\
		Coefficient extraction & Cauchy's Integral Formula\\
		Integral simplification & Orthogonality identity\\
		Summation & Multi-index combinatorics\\
		Final estimate & Membership condition\\
		Sharpness & Extremal mapping\\
		\hline
	\end{tabular}
\end{table}

\end{proof}

\begin{lem}\label{Lm-3.1}\emph{\cite[Theorem 6.1.4]{Graham-Kohr}} Let $f$ be holomorphic in the polydisk $\mathbb{P}\Delta(0_n;1_n)$ such that $|f(z)|\leq 1$ for all $z\in \mathbb{P}\Delta(0_n;1_n)$. Then
\begin{align*}
\left|\frac{\partial^{|\beta|} f(0)}{\partial z_1^{\beta_1}\ldots \partial z_n^{\beta_n}}\right|\leq \beta!
\end{align*}
 for multi-index $\beta=(\beta_1,\ldots, \beta_n)$.
\end{lem}
Let $G\not=\varnothing$ be an open subset of $\mathbb{C}^n$. Let $f$ be a holomorphic function on $G$. For a point $a\in\mathbb{C}^n$, we write 
\begin{align*}
	f(z)=\sum_{i=0}^{\infty}P_i(z-a),
\end{align*}
where the term $P_i(z-a)$ is either identically zero or a homogeneous polynomial of degree $i$. Denote the zero-multiplicity of $f$ at $a$ by 
\begin{align*}
	k=\min\{i:P_i(z-a)\not\equiv 0\}.
\end{align*}
Clearly $1$ is the zero-multiplicity of $f$ at $a$ when $f(a)=0$ and $\frac{\partial f(a)}{\partial z_j}\neq 0$ for some $j=1,2,\ldots,n$.

New we state the multidimensional version of Theorem C ($(i)$ and $(ii)$).
\begin{theo}\label{Th-1.3} Let $f=h+\ol g\in \mathcal{W}_{\mathcal{H}_n^0}(\alpha)$ and be given by (\ref{Eq 1.7}). Then for any multi-index $\beta=(\beta_1,\beta_2,\ldots,\beta_n)$ such that $|\beta|=m\geq 2$, we have
\begin{enumerate}
\item[\emph{(i)}] $\displaystyle \left|\sum\limits_{|\beta|=m}a_{\beta}\right|+\left|\sum\limits_{|\beta|=m}b_{\beta}\right|\leq \frac{2n^2\binom{m+n-1}{n-1}}{\alpha m^2+(n-\alpha)m}$,\vspace{1.2mm}
\item[\emph{(ii)}] $\displaystyle \left|\;\left|\sum\limits_{|\beta|=m}a_{\beta}\right|-\left|\sum\limits_{|\beta|=m}b_{\beta}\right|\;\right|\leq \frac{2n^2\binom{m+n-1}{n-1}}{\alpha m^2+(n-\alpha)m}$,\vspace{1.2mm}
\item[\emph{(iii)}] 
$\displaystyle \left|\sum\limits_{|\beta|=m}a_{\beta}\right|\leq \frac{2n^2\binom{m+n-1}{n-1}}{\alpha m^2+(n-\alpha)m}$.
\end{enumerate}
All three inequalities are sharp.
\end{theo}

\begin{proof}Let $f=h+\ol g\in \mathcal{W}_{\mathcal{H}_n^0}(\alpha)$. Then from Theorem \ref{Th-1.1}, we see that the function $F_{\varepsilon}=h+\varepsilon g$ belongs to $\mathcal{W}_{n}(\alpha)$ for each $\varepsilon\; (|\varepsilon|=1)$, and further
\begin{align}\label{Th3:1.1}
\displaystyle\Re\left(\frac{\partial F_{\varepsilon}(z)}{\partial z_k}+\alpha z_l\frac{\partial^2 F_{\varepsilon}(z)}{\partial z_j\partial z_k}\right)
=\Re\left(\frac{\partial h(z)}{\partial z_k}+\alpha z_l\frac{\partial^2 h(z)}{\partial z_j\partial z_k}+\varepsilon\left(\frac{\partial g(z)}{\partial z_k}+\alpha z_l\frac{\partial^2 g(z)}{\partial z_j\partial z_k}\right)\right)>0
\end{align}
for all $z=(z_1,z_2,\ldots,z_n)\in \mathbb{P} \Delta(0_n;1_n)$ and for all $j,k,l\in\mathbb{Z}[1,n]$. It is clear that
\begin{align*}
\displaystyle F_{\varepsilon}(z)=\sum\limits_{j=1}^n z_j+\sum\limits_{m=2}^{\infty}\left(P_m(z)+\varepsilon Q_m(z)\right),
\end{align*}
for all $z\in \mathbb{P} \Delta(0_n;1_n)$, where $P_m(z)$ and $Q_m(z)$ are defined in (\ref{Th2:1.3}). Note that
 \begin{align*}
\displaystyle \frac{\partial F_{\varepsilon}(z)}{\partial z_k}=&1+
 \sum\limits_{m=2}^{\infty} \frac{\partial \left(P_m(z)+\varepsilon Q_m(z)\right)}{\partial z_k}\\=&1+\sum\limits_{m=2}^{\infty}\sum\limits_{|\beta|=m}\beta_{k} c_{\beta}z_1^{\beta_1}\ldots z_{j-1}^{\beta_{j-1}}z_{j}^{\beta_j}z_{j+1}^{\beta_{j+1}}\ldots z_{k-1}^{\beta_{k-1}}z_{k}^{\beta_k-1}z_{k+1}^{\beta_{k+1}}\ldots z_l^{\beta_l}\ldots z_n^{\beta_n}
 \end{align*}
 and
\begin{align*}
\displaystyle z_l\frac{\partial^2 F_{\varepsilon}(z)}{\partial z_j\partial z_k}=&
\sum\limits_{m=2}^{\infty} z_l\frac{\partial^2 \left(P_m(z)+\varepsilon Q_m(z)\right)}{\partial z_j\partial z_k}\\=&\sum\limits_{m=2}^{\infty}\sum\limits_{|\beta|=m}\beta_{jk} c_{\beta}z_1^{\beta_1}\ldots z_{j-1}^{\beta_{j-1}}z_{j}^{\beta_j^*}z_{j+1}^{\beta_{j+1}}\ldots z_{k-1}^{\beta_{k-1}}z_{k}^{\beta_k^{**}}z_{k+1}^{\beta_{k+1}}\ldots z_l^{\beta_l+1}\ldots z_n^{\beta_n},\nonumber
\end{align*}
where $c_{\beta}=a_{\beta}+\varepsilon b_{\beta}$ and $\beta_{jk}$, $\beta^*_j$ and $\beta^{**}_k$ are defined as in (\ref{ll.1}) and (\ref{ll.2}) respectively.
Consequently
\begin{align}\label{Th3:1.2}
\displaystyle &\frac{\partial F_{\varepsilon}(z)}{\partial z_k}+\alpha z_l\frac{\partial^2 F_{\varepsilon}(z)}{\partial z_j\partial z_k}\\ =&1+
\sum\limits_{m=2}^{\infty} \frac{\partial \left(P_m(z)+\varepsilon Q_m(z)\right)}{\partial z_k}+\alpha \sum\limits_{m=2}^{\infty} z_l\frac{\partial^2 \left(P_m(z)+\varepsilon Q_m(z)\right)}{\partial z_j\partial z_k}\nonumber\\=&
1+\sum\limits_{m=2}^{\infty}\sum\limits_{|\beta|=m}\beta_{k} c_{\beta}z_1^{\beta_1}\ldots z_{j-1}^{\beta_{j-1}}z_{j}^{\beta_j}z_{j+1}^{\beta_{j+1}}\ldots z_{k-1}^{\beta_{k-1}}z_{k}^{\beta_k-1}z_{k+1}^{\beta_{k+1}}\ldots z_l^{\beta_l}\ldots z_n^{\beta_n}\nonumber\\&+
\alpha\sum\limits_{m=2}^{\infty}\sum\limits_{|\beta|=m}\beta_{jk} c_{\beta}z_1^{\beta_1}\ldots z_{j-1}^{\beta_{j-1}}z_{j}^{\beta_j^*}z_{j+1}^{\beta_{j+1}}\ldots z_{k-1}^{\beta_{k-1}}z_{k}^{\beta_k^{**}}z_{k+1}^{\beta_{k+1}}\ldots z_l^{\beta_l+1}\ldots z_n^{\beta_n}\nonumber
\end{align}
for all $z=(z_1,z_2,\ldots,z_n)\in \mathbb{P} \Delta(0_n;1_n)$ and for all $j,k,l\in\mathbb{Z}[1,n]$.

Let 
\begin{align}\label{Th3:1.3}
\displaystyle P_{jkl}(z)=\frac{\partial F_{\varepsilon}(z)}{\partial z_k}+\alpha z_l\frac{\partial^2 F_{\varepsilon}(z)}{\partial z_j\partial z_k}.
\end{align}
for all $z=(z_1,z_2,\ldots,z_n)\in \mathbb{P} \Delta(0_n;1_n)$ and for all $j,k,l\in\mathbb{Z}[1,n]$.

 Clearly $P_{jkl}(z)$ is holomorphic in $\mathbb{P}\Delta(0_n;1_n)$. Using (\ref{Th3:1.1}) to (\ref{Th3:1.3}), we see that $\Re\{P_{jkl}(z)\} > 0$. Also from \eqref{Th3:1.2} and \eqref{Th3:1.3},  we have $P_{jkl}(0)=1$.
 
 Now, we may assume that
 \begin{align}\label{Th3:1.4}
 \displaystyle	P_{jkl}(z)=1+\sum\limits_{m=1}^{\infty}\sum\limits_{|\beta|=m}P_{\beta}(j,k,l)z^{\beta}
 \end{align}
 for all $z\in \mathbb{P}\Delta(0_n;1_n)$ and for all $j,k,l\in\mathbb{Z}[1,n]$.
 
 Let $\omega_l=e^{\frac{2\pi \iota}{k}l}$, where $k\geq 1$ is an integer. We see that $\sum_{l=1}^k \omega_l^s=0$ if $s\leq k-1$. Let
 \begin{align*}
 \displaystyle	g_{jkl}(z)=\frac{1}{k}\sum\limits_{l=1}^k P_{jkl}\left(\omega_l z\right)
 \end{align*}
 for all $j,k,l\in\mathbb{Z}[1,n]$.
 The function $g_{jkl}$ is clearly holomorphic in $\mathbb{P}\Delta(0_n;1_n)$ such that $\Re\{g_{jkl}(z)\} > 0$ throughout the domain. Furthermore, we see that $g_{jkl}(0) =1$ and $g_{jkl}$ has the series expansion
 \begin{align*}
 \displaystyle	g_{jkl}(z)=1+\sum\limits_{m=1}^{\infty}\sum\limits_{|\beta|=mk} P_{\beta}(j,k,l)z^{\beta}
 \end{align*}
 for all $z\in \mathbb{P}\Delta(0_n;1_n)$.
 
 Let
 \begin{align}
 \displaystyle	\label{Eq-3.20}\phi_{jkl}(z)=\frac{g_{jkl}(z)-1}{g_{jkl}(z)+1}
 \end{align}
 for all $z\in \mathbb{P}\Delta(0_n;1_n)$. It is easy to verify that $|\phi_{jkl}(z)|<1$ for all $z\in \mathbb{P}\Delta(0_n;1_n)$ and $k$ is the zero-multiplicity of $\phi_{jkl}$ at $0$. We expand $\phi_{jkl}(z)$ to a Taylor series with respect to multi-index $\beta$,
 \begin{align}
 	\label{Eq-3.21} \phi_{jkl}(z)=\sum\limits_{m=k}^{\infty}\sum\limits_{|\beta|=m}Q_{\beta}(j,k,l)z^{\beta}.
 \end{align}
 
 By Lemma \ref{Lm-3.1}, we have
 \begin{align}
 \displaystyle	\label{Eq-3.22} |Q_{\beta}(j,k,l)|\leq 1
 \end{align}
 for multi-index $\beta$. 
 
 Now from \eqref{Eq-3.20} and \eqref{Eq-3.21}, we deduce that
 \begin{align*}
\displaystyle &\sum\limits_{|\beta|=k} P_{\beta}(j,k,l)z^{\beta}+\sum\limits_{m=2}^{\infty}\sum\limits_{|\beta|=mk} P_{\beta}(j,k,l)z^{\beta}\\=&2 \sum\limits_{m=k}^{\infty}\sum\limits_{|\beta|=m}Q_{\beta}(j,k,l)z^{\beta}\\&+
 	\left(\sum\limits_{|\beta|=k} P_{\beta}(j,k,l)z^{\beta}+\sum\limits_{m=2}^{\infty}\sum\limits_{|\beta|=mk} P_{\beta}(j,k,l)z^{\beta}\right)\left(\sum\limits_{m=k}^{\infty}\sum\limits_{|\beta|=m}Q_{\beta}(j,k,l)z^{\beta}\right).
 \end{align*}
 
 It then follows that
 \begin{align}\label{Eq-3.23} \displaystyle \sum\limits_{|\beta|=k} \left(P_{\beta}(j,k,l)-2Q_{\beta}(j,k,l)\right)z^{\beta}+\text{higher degree terms of}\; (z_1^{\beta_1},\ldots,z_n^{\beta_n})\equiv 0
 \end{align}
 holds for all $z\in \mathbb{P}\Delta(0_n;1_n)$. 
 
 Then by a simple calculation, we get from \eqref{Eq-3.23} that 
 \begin{align}
 \displaystyle	\label{Eq-3.24}\sum\limits_{|\beta|=k} \left(P_{\beta}(j,k,l)-2Q_{\beta}(j,k,l)\right)=0.
 \end{align}
 \noindent Using \eqref{Eq-3.22} and \eqref{Eq-3.24}, we conclude that
 \begin{align}\label{Th3:1.5}
 \displaystyle	\sum\limits_{|\beta|=k} |P_{\beta}(j,k,l)|\leq 2\sum\limits_{|\beta|=k} |Q_{\beta}(j,k,l)|\leq 2\sum\limits_{|\beta|=k} 1\leq 2\binom{k+n-1}{n-1},
 \end{align}
 where $\beta=(\beta_1,\beta_2,\ldots,\beta_n)$ such that $|\beta|=k$ and for all $j,k,l\in\mathbb{Z}[1,n]$.\vspace{1.2mm}


On the other hand, from \eqref{Th3:1.2} and \eqref{Th3:1.4}, we obtain
\begin{align}\label{Th3:1.6}
\displaystyle &\sum\limits_{m=2}^{\infty}\sum\limits_{|\beta|=m}\sum\limits_{l=1}^n\sum\limits_{j=1}^n\sum\limits_{k=1}^n\beta_{k} c_{\beta}z_1^{\beta_1}\ldots z_{j-1}^{\beta_{j-1}}z_{j}^{\beta_j}z_{j+1}^{\beta_{j+1}}\ldots z_{k-1}^{\beta_{k-1}}z_{k}^{\beta_k-1}z_{k+1}^{\beta_{k+1}}\ldots z_l^{\beta_l}\ldots z_n^{\beta_n}\\&+
\alpha\sum\limits_{m=2}^{\infty}\sum\limits_{|\beta|=m}\sum\limits_{l=1}^n\sum\limits_{j=1}^n\sum\limits_{k=1}^n\beta_{jk} c_{\beta}z_1^{\beta_1}\ldots z_{j-1}^{\beta_{j-1}}z_{j}^{\beta_j^*}z_{j+1}^{\beta_{j+1}}\ldots z_{k-1}^{\beta_{k-1}}z_{k}^{\beta_k^{**}}z_{k+1}^{\beta_{k+1}}\ldots z_l^{\beta_l+1}\ldots z_n^{\beta_n}\nonumber\\=&\sum\limits_{m=1}^{\infty}\sum\limits_{|\beta|=m}\sum\limits_{l=1}^n\sum\limits_{j=1}^n\sum\limits_{k=1}^nP_{\beta}(j,k,l)z^{\beta}\nonumber,
\end{align}
where $\beta=(\beta_1,\beta_2,\ldots,\beta_n)$ such that $|\beta|=m\geq 2$.

 Putting $z_1=z_2=\ldots=z_n=z$ on both sides of \eqref{Th3:1.6} and then comparing the coefficients, we get
\begin{align*}
\displaystyle \sum\limits_{|\beta|=m}\sum\limits_{l=1}^n\sum\limits_{j=1}^n\sum\limits_{k=1}^n
\left(\beta_k+\alpha\beta_{jk}\right)c_{\beta}=\sum\limits_{|\beta|=m}\sum\limits_{l=1}^n\sum\limits_{j=1}^n\sum\limits_{k=1}^nP_{\beta}(j,k,l),
\end{align*}
for all $\beta=(\beta_1,\beta_2,\ldots,\beta_n)$ such that $|\beta|=m\geq 2$. Consequently
\begin{align}\label{Th3:1.7}
&\sum\limits_{|\beta|=m}\sum\limits_{l=1}^n\sum\limits_{j=1}^n\sum\limits_{k=1}^n
\left(\beta_k+\alpha\beta_{jk}\right)a_{\beta}+\varepsilon \sum\limits_{|\beta|=m}\sum\limits_{l=1}^n\sum\limits_{j=1}^n\sum\limits_{k=1}^n
\left(\beta_k+\alpha\beta_{jk}\right)b_{\beta} \\=&\sum\limits_{|\beta|=m}\sum\limits_{l=1}^n\sum\limits_{j=1}^n\sum\limits_{k=1}^nP_{\beta}(j,k,l)\nonumber
\end{align}	
for all $\beta=(\beta_1,\beta_2,\ldots,\beta_n)$ such that $|\beta|=m\geq 2$. 

Now in view of \eqref{Th2:1.8a}, we conclude that
\begin{align}\label{Th3:1.8}
\displaystyle	\sum\limits_{|\beta|=m}\sum\limits_{l=1}^n\sum\limits_{j=1}^n\sum\limits_{k=1}^n\beta_{jk}a_{\beta}=nm(m-1)\sum\limits_{|\beta|=m}a_{\beta}
\end{align}
and
\begin{align}\label{Th3:1.9}
\displaystyle	\sum\limits_{|\beta|=m}\sum\limits_{l=1}^n\sum\limits_{j=1}^n\sum\limits_{k=1}^n\beta_{jk}b_{\beta}=nm(m-1)\sum\limits_{|\beta|=m}b_{\beta}
\end{align}
for all $\beta=(\beta_1,\beta_2,\ldots,\beta_n)$ such that $|\beta|=m\geq 2$. 

Using \eqref{Th3:1.8} and \eqref{Th3:1.9} to \eqref{Th3:1.7}, we find that
\begin{align}\label{Th3:1.10}
\displaystyle &(n^2m+\alpha nm(m-1))\sum\limits_{|\beta|=m}a_{\beta}+\varepsilon (n^2m+\alpha nm(m-1))\sum\limits_{|\beta|=m} b_{\beta}\\=&\sum\limits_{|\beta|=m}\sum\limits_{l=1}^n\sum\limits_{j=1}^n\sum\limits_{k=1}^nP_{\beta}(j,k,l)\nonumber
\end{align}
for all $\beta=(\beta_1,\beta_2,\ldots,\beta_n)$ such that $|\beta|=m\geq 2$. 

Applying \eqref{Th3:1.5} to \eqref{Th3:1.10}, we obtain
\begin{align}\label{Th3:1.11}
\displaystyle \left|\sum\limits_{|\beta|=m}a_{\beta}+\varepsilon \sum\limits_{|\beta|=m} b_{\beta}\right|\leq 
\frac{2n^2\binom{m+n-1}{n-1}}{\alpha m^2+(n-\alpha)m}
\end{align}
for every integer $m\ge2$ and every
$\varepsilon$ with $|\varepsilon|=1$. \vspace{1.2mm}
	
Suppose $A_{\beta}=\sum\limits_{|\beta|=m}a_{\beta}\neq 0$ and $B_{\beta}=\sum\limits_{|\beta|=m}b_{\beta}\neq 0$. If we choose 
\begin{align*}
\displaystyle \varepsilon=\frac{A_{\beta}}{|A_{\beta}|}\frac{\ol B_{\beta}}{|B_{\beta}|},
\end{align*}
then $|\varepsilon|=1$ and so from \eqref{Th3:1.11}, we have
\begin{align}\label{Th3:1.12}
\displaystyle |A_{\beta}+\varepsilon B_{\beta}|=|A_{\beta}|+|B_{\beta}|\leq \frac{2n^2\binom{m+n-1}{n-1}}{\alpha m^2+(n-\alpha)m}
\end{align}
for $\beta=(\beta_1,\beta_2,\ldots,\beta_n)$ such that $|\beta|=m\geq 2$. If either $A_{\beta}=0$ or $B_{\beta}=0$, then \eqref{Th3:1.12} also holds.
Therefore 
\begin{align}\label{Th3:1.13}
	\left|\sum\limits_{|\beta|=m}a_{\beta}\right|+\left|\sum\limits_{|\beta|=m}b_{\beta}\right|\leq \frac{2n^2\binom{m+n-1}{n-1}}{\alpha m^2+(n-\alpha)m},
\end{align}
for every integer $m\ge2$.

 On the other hand, we have
\begin{align}\label{Th3:1.14}
\left||A_{\beta}|-|B_{\beta}|\right|\leq |A_{\beta}+\varepsilon B_{\beta}|\leq \frac{2n^2\binom{m+n-1}{n-1}}{\alpha m^2+(n-\alpha)m}
\end{align}
for each integer $m\ge2$.

Therefore from (\ref{Th3:1.14}), we obtain
 \begin{align*}
\displaystyle \left|\;\left|\sum\limits_{|\beta|=m}a_{\beta}\right|-\left|\sum\limits_{|\beta|=m}b_{\beta}\right|\;\right|\leq \frac{2n^2\binom{m+n-1}{n-1}}{\alpha m^2+(n-\alpha)m},
\end{align*}
for every integer $m\ge2$.

It is clear from \eqref{Th3:1.13} that 
\begin{align*}
\displaystyle \left|\sum\limits_{|\beta|=m}a_{\beta}\right|\leq \dfrac{2n^2\binom{m+n-1}{n-1}}{\alpha m^2+(n-\alpha)m},
\end{align*}
for every integer $m\ge2$.\vspace{1.2mm}

To show that all the inequalities are sharp, we consider the following function
\begin{align}\label{Ex1}
\displaystyle f_2(z)=\sum\limits_{j=1}^n z_j+\sum\limits_{m=2}^{\infty}\sum\limits_{|\beta|=m} b_{\beta} z^{\beta},
\end{align}
where 
\begin{align*}
b_{\beta}=\frac{2n^2}{\alpha m^2+(n-\alpha)m}
\end{align*}
for all multi-index $\beta=(\beta_1,\beta_2,\ldots,\beta_n)$ such that $|\beta|=m$. 
It is easy to see that $f_2\in \mathcal{W}_{\mathcal{H}_n^0}(\alpha)$, and 
\begin{align*}
\displaystyle \left|\sum\limits_{|\beta|=m}a_{\beta}(f_2)\right|=\frac{2\binom{m+n-1}{n-1}n^2}{\alpha m^2+(n-\alpha)m}.
\end{align*} 


\end{proof}

In the following result, we have established the growth estimates for the functions in the class $\mathcal{W}_{\mathcal{H}_n^0}(\alpha)$.

\begin{theo}\label{Th-1.4}Let $f=h+\ol g\in \mathcal{W}_{\mathcal{H}_n^0}(\alpha)$ with $0<\alpha\leq 1$ and be given by (\ref{Eq 1.7}). Then for $z\in \mathbb{P} \Delta(0;1/n)$, we have
\begin{align*}
\displaystyle \left|\sum\limits_{j=1}^nz_j\right|+2\sum\limits_{m=2}^{\infty} \frac{(-1)^{m-1}n^m||z||_{\infty}^{m}}{\alpha m^2+(1-\alpha)m}\leq |f(z)|\leq n||z||_{\infty}+2\sum\limits_{m=2}^{\infty} \frac{n^m||z||_{\infty}^{m}}{\alpha m^2+(1-\alpha)m}.
\end{align*}
Both inequalities are sharp.
\end{theo}

\begin{proof} Let $f=h+\ol g\in \mathcal{W}_{\mathcal{H}_n^0}(\alpha)$. Then from Theorem \ref{Th-1.1}, we see that the function $F_{\varepsilon}=h+\varepsilon g$ belongs to $\mathcal{W}_{n}(\alpha)$ for each $\varepsilon\; (|\varepsilon|=1)$, and further 
\begin{align*}
\Re\left(\frac{\partial F_{\varepsilon}(z)}{\partial z_k}+\alpha z_l\frac{\partial^2 F_{\varepsilon}(z)}{\partial z_j\partial z_k}\right)>0	
\end{align*}
for all $z=(z_1,z_2,\ldots,z_n)\in \mathbb{P} \Delta(0;1/n)$ and for all $j,k,l\in\mathbb{Z}[1,n]$. Then there exists a holomorphic function $\omega_{jkl}(z)$ with $\omega_{jkl}(0)=0$ and $|\omega_{jkl}(z)|<1$ for all $z=(z_1,z_2,\ldots,z_n)\in \mathbb{P} \Delta(0;1/n)$ such that 
\begin{align}\label{Th4:1.1}	
\frac{\partial F_{\varepsilon}(z)}{\partial z_k}+\alpha z_l\frac{\partial^2 F_{\varepsilon}(z)}{\partial z_j\partial z_k}=\frac{1+\omega_{jkl}(z)}{1-\omega_{jkl}(z)}
\end{align}
for all $z=(z_1,z_2,\ldots,z_n)\in \mathbb{P} \Delta(0;1/n)$, where $j,k,l\in\mathbb{Z}[1,n]$.

Let $z\in \mathbb{P}\Delta(0;1/n)$ be fixed in such a way that $z\neq 0$. Obviously $\|z\|_{\infty}<1/n$. 
We define 
\begin{align}\label{Th4:1.2}
\tilde h(t)=h\left(\frac{z}{n||z||_{\infty}}t\right)=\frac{1}{n||z||_{\infty}}\left(\sum\limits_{j=1}^n z_j\right)t+\sum\limits_{m=2}^{\infty}P_m\left(\frac{z}{n||z||_{\infty}}\right)t^m
\end{align}
and
\begin{align}\label{Th4:1.3}
\tilde g(t)=g\left(\frac{z}{n||z||_{\infty}}t\right)=\sum\limits_{k=2}^{\infty} Q_m\left(\frac{z}{n||z||_{\infty}}\right)t^m,
\end{align}
where $h(z)$ and $g(z)$ are defined by (\ref{Th2:1.2}) and $t\in\mathbb{C}$ such that $|t|<1$. For a fixed $z\in \mathbb{P}\Delta(0;1/n)$, we can say that $\tilde g(t)$ and $\tilde h(t)$ are analytic in $|t|<1$. Obviously for a fixed $z\in \mathbb{P}\Delta(0;1/n)$, the following function
\begin{align}\label{Th4:1.4}
\tilde F_{\varepsilon}(t)=F_{\varepsilon}\left(\frac{z}{n||z||_{\infty}}t\right)=\tilde h(t)+\varepsilon \tilde g(t)
\end{align}
is analytic in $|t|<1$, where $\tilde h(t)$ and $\tilde g(t)$ are defined by (\ref{Th4:1.2}) and (\ref{Th4:1.3}). Also from (\ref{Th4:1.2}) and (\ref{Th4:1.3}), we see that
\begin{align}\label{Th4:1.5}
\tilde f(t)=f\left(\frac{z}{n||z||_{\infty}}t\right)=\tilde h(t)+\ol{\tilde g(t)}.
\end{align}

We define 
\begin{align}\label{Th4:1.6}
\tilde \omega_{jkl}(t)=\omega_{jkl}\left(\frac{z}{n||z||_{\infty}}t\right),
\end{align}
which is analytic in $|t|<1$, where $j,k,l\in\mathbb{Z}[1,n]$.
In terms of one variable $t$, we deduce from (\ref{Th4:1.1}) that
\begin{align}\label{Th4:1.7}
\frac{d \tilde F_{\varepsilon}(t)}{d t}+\alpha t\frac{d^2 \tilde F_{\varepsilon}(t)}{d t^2}=\frac{1+\tilde \omega_{jkl}(t)}{1-\tilde \omega_{jkl}(t)}.
\end{align}

Applying the argument used in the proof of Theorem 4.4 of \cite{Ghosh-Allu-2019}, we obtain 
\begin{align}\label{Th4:1.8}
\left|\frac{d \tilde F_{\varepsilon}(t)}{d t}\right|=\left|\frac{d \tilde h(t)}{d t}+\varepsilon \frac{d \tilde g(t)}{d t}\right|\leq 1+2\sum\limits_{k=1}^{\infty}\frac{|t|^{k}}{1+k\alpha}
\end{align}
and
\begin{align}\label{Th4:1.9}
\left|\frac{d \tilde F_{\varepsilon}(t)}{d t}\right|=\left|\frac{d \tilde h(t)}{d t}+\varepsilon \frac{d \tilde g(t)}{d t}\right|\geq 1+2\sum\limits_{k=1}^{\infty}\frac{(-1)^k|t|^{k}}{1+k\alpha}.
\end{align}

Since $\varepsilon\;(|\varepsilon|=1)$ is arbitrary, it follows from (\ref{Th4:1.8}) and (\ref{Th4:1.9}) that
\begin{align}\label{Th4:1.10}
\left|\frac{d \tilde h(t)}{d t}\right|+\left|\frac{d \tilde g(t)}{d t}\right|\leq 1+2\sum\limits_{k=1}^{\infty}\frac{|t|^{k}}{1+k\alpha}
\end{align}
and
\begin{align}\label{Th4:1.11}
\left|\frac{d \tilde h(t)}{d t}\right|-\left|\frac{d \tilde g(t)}{d t}\right|\geq 1+2\sum\limits_{k=1}^{\infty}\frac{(-1)^k|t|^{k}}{1+k\alpha}.
\end{align}

Let $\Gamma$ be the radial segment from $0$ to $n||z||_{\infty}$. Then from (\ref{Th4:1.5}) and (\ref{Th4:1.10}), we have
\begin{align}\label{Th4:1.12}
|\tilde f(n\|z\|_{\infty})|=\left|\int_{\Gamma} \frac{\partial \tilde f(s)}{\partial s}d s+\frac{\partial \tilde f(s)}{\partial \ol s}d \ol s\right|\leq&
\int\limits_{0}^{n||z||_{\infty}}\left\lbrack \left|\frac{d \tilde h(s)}{d s}\right|+\left|\frac{d \tilde g(s)}{d s}\right|\right\rbrack |d s|\\\leq&
\int\limits_{0}^{n||z||_{\infty}}\left(1+2\sum\limits_{k=1}^{\infty}\frac{\xi^{k}}{1+k\alpha}\right)d\xi\nonumber\\=&
n||z||_{\infty}+2\sum\limits_{k=1}^{\infty} \frac{n^{k+1}||z||_{\infty}^{k+1}}{(k+1)(1+k\alpha)}\nonumber\\=&
n||z||_{\infty}+2\sum\limits_{m=2}^{\infty} \frac{n^m||z||_{\infty}^{m}}{\alpha m^2+(1-\alpha)m}.\nonumber
\end{align}
\begin{figure}[t]
	\centering
	
	\begin{tikzpicture}
		
		\begin{axis}[
			width=0.82\textwidth,
			height=7cm,
			xmin=0,
			xmax=1,
			ymin=0,
			ymax=2.4,
			xlabel={$r=n|z|_{\infty}$},
			ylabel={Bound value},
			grid=major,
			legend style={
				draw=black,
				fill=white,
				font=\small,
				at={(0.03,0.97)},
				anchor=north west
			},
			tick label style={font=\small},
			label style={font=\small},
			every axis plot/.append style={
				very thick,
				smooth,
				samples=200
			}
			]
			
			\addplot[blue]
			{1.75*x+0.55*x^2};
			
			\addlegendentry{Upper bound}
			
			\addplot[red,dashed]
			{x+0.18*x^2-0.22*x^3};
			
			\addlegendentry{Lower bound}
			
		\end{axis}
		
	\end{tikzpicture}
	
	\caption{Schematic illustration of the growth bounds in
		Theorem~\ref{Th-1.4}.}
	
	\label{fig:GrowthBoundTh14}
	
\end{figure}
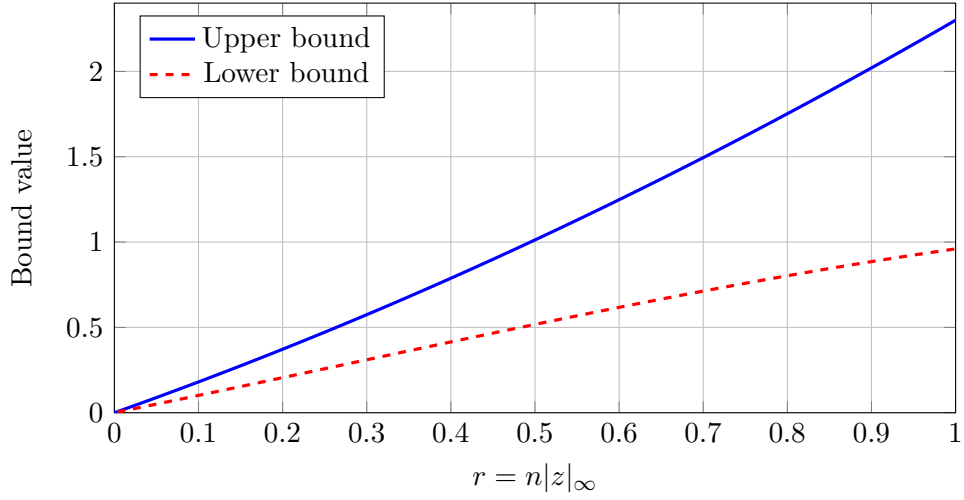

Consequently from (\ref{Th4:1.5}) and (\ref{Th4:1.12}), we get
\begin{align*}
\displaystyle |f(z)|\leq n||z||_{\infty}+2\sum\limits_{m=2}^{\infty} \frac{n^m||z||_{\infty}^{m}}{\alpha m^2+(1-\alpha)m},
\end{align*}
which also holds for $z=0$.

Again from (\ref{Th4:1.5}) and (\ref{Th4:1.11}), we have
\begin{align*}
\displaystyle |\tilde f(n\|z\|_{\infty})|=&\left|\int_{\Gamma} \frac{\partial \tilde f(s)}{\partial s}d s+\frac{\partial \tilde f(s)}{\partial \ol s}d \ol s\right|\\\geq&
\int\limits_{0}^{n||z||_{\infty}}\left\lbrack \left|\frac{d \tilde h(s)}{d s}\right|-\left|\frac{d \tilde g(s)}{d s}\right|\right\rbrack |d s|\nonumber\\\geq&
n||z||_{\infty}+2\sum\limits_{m=2}^{\infty} \frac{(-1)^{m-1}n^m||z||_{\infty}^{m}}{\alpha m^2+(1-\alpha)m}.\nonumber
\end{align*}
Consequently
\begin{align*}
\displaystyle |f(z)|\geq n||z||_{\infty}+2\sum\limits_{m=2}^{\infty} \frac{(-1)^{m-1}n^m||z||_{\infty}^{m}}{\alpha m^2+(1-\alpha)m},
\end{align*}
which also holds for $z=0$. 

Since $|\sum_{j=1}^n z_j|\leq n\|z\|_{\infty}$, we have
\begin{align*}
	\displaystyle |f(z)|\geq \left|\sum\limits_{j=1}^nz_j\right|+2\sum\limits_{m=2}^{\infty} \frac{(-1)^{m-1}n^m||z||_{\infty}^{m}}{\alpha m^2+(1-\alpha)m}.
\end{align*}

Thus the desired inequalities are established.\vspace{2mm}



To show that the bounds are sharp, we consider the following functions
\begin{align*}
\displaystyle f_4(z)=\sum\limits_{j=1}^n z_j+\sum\limits_{m=2}^{\infty}\sum\limits_{|\beta|=m}a_{\beta}z^{\beta}
\end{align*}
and
\begin{align*} 
\displaystyle f_5(z)=\sum\limits_{j=1}^n z_j+\sum\limits_{m=2}^{\infty}\sum\limits_{|\beta|=m} b_{\beta}z^{\beta},
\end{align*}
where
\begin{align*}
\displaystyle a_{\beta}=\frac{2n^m}{\binom{m+n-1}{n-1}(\alpha m^2+(1-\alpha)m)}
\end{align*}
and
\begin{align*}
\displaystyle	b_{\beta}=\frac{2(-1)^{m-1}n^m}{\binom{m+n-1}{n-1}(\alpha m^2+(1-\alpha)m)}
\end{align*}
for every integer $m\geq 2$. 

It is easy to verify that $f_4,f_5\in \mathcal{W}_{\mathcal{H}_n^0}(\alpha)$.\vspace{1.2mm}

Now for the point $z=(r,r,\ldots,r)$, where $r<1/n$, we find that
\begin{align*}
\displaystyle |f_4(z)|=n||z||_{\infty}+\sum\limits_{m=2}^{\infty} \frac{2n^m||z||_{\infty}^{m}}{\alpha m^2+(1-\alpha)m}
\end{align*}
and
\begin{align*}
\displaystyle |f_5(z)|=\left|\sum\limits_{j=1}^n z_j\right|+\sum\limits_{m=2}^{\infty} \frac{2(-1)^{m-1}n^m||z||_{\infty}^{m}}{\alpha m^2+(1-\alpha)m}.
\end{align*}
	\begin{center}
	\begin{tabular}{cc}
	\includegraphics[width=\textwidth]{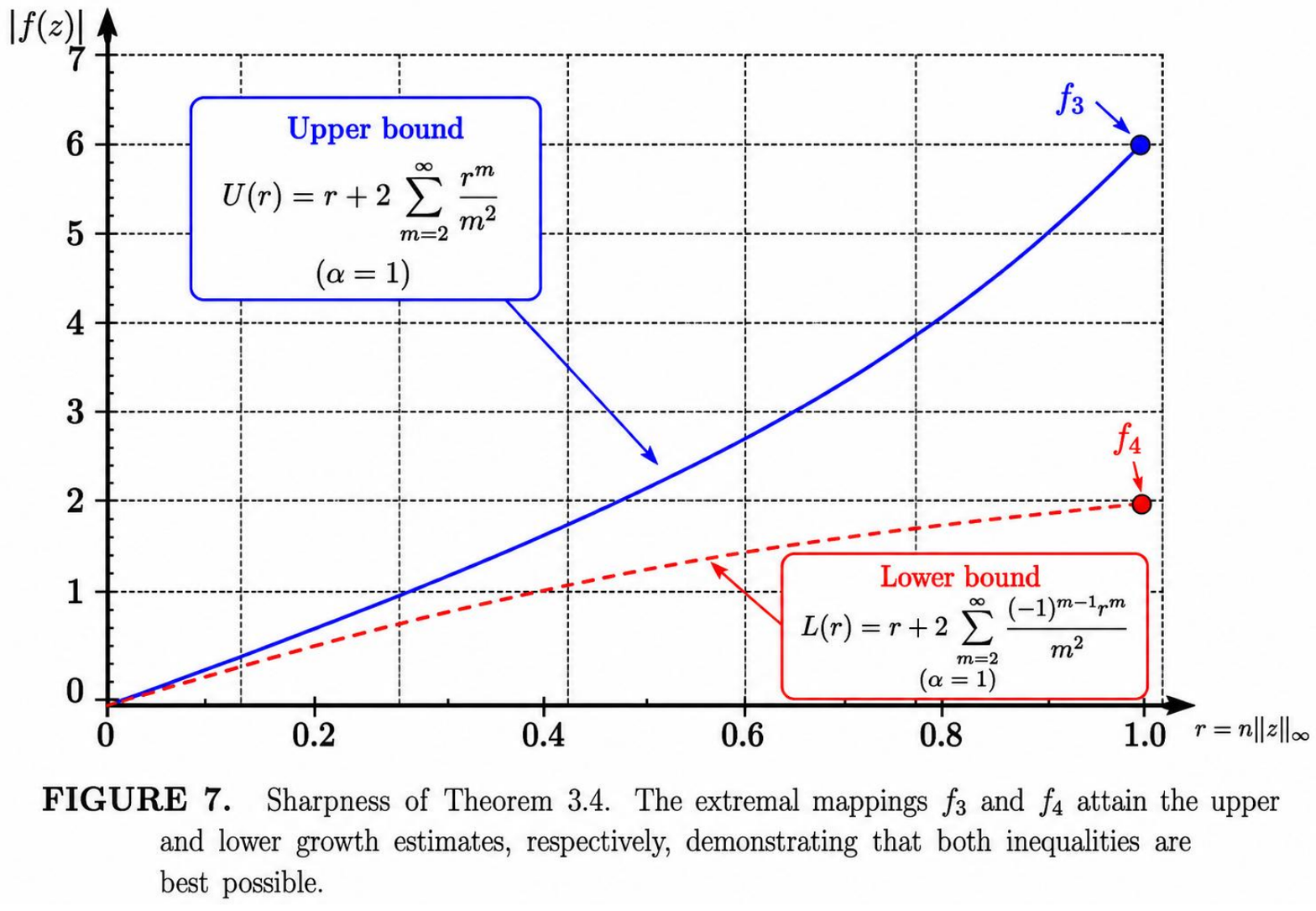}\\ 
	\end{tabular}
\end{center}
\end{proof}
			
\begin{theo}\label{Th-1.5}Let $f=h+\ol g\in \mathcal{H}_n^0$ and be given by (\ref{Eq 1.7}). If
\begin{align}\label{Th5:1.0}
\sum\limits_{m=2}^{\infty}\sum\limits_{|\beta|=m}n\left(m^2\alpha+(n-\alpha)m\right)\left(|a_{\beta}|+|b_{\beta}|\right)<1,
\end{align}
then $f\in \mathcal{W}_{\mathcal{H}_n^0}(\alpha)$.
\end{theo}

\begin{proof} Let $f=h+\ol g\in \mathcal{H}_n^0$. Now from (\ref{Th2:1.8c}), we have
\begin{align}\label{Th5:1.1}
&\Re\left(\frac{\partial h(z)}{\partial z_k}+\alpha z_l\frac{\partial^2 h(z)}{\partial z_j\partial z_k}\right)\\=&1+
\Re\left(\sum\limits_{m=2}^{\infty}\sum\limits_{|\beta|=m}\beta_{k} a_{\beta}z_1^{\beta_1}\ldots z_{j-1}^{\beta_{j-1}}z_{j}^{\beta_j}z_{j+1}^{\beta_{j+1}}\ldots z_{k-1}^{\beta_{k-1}}z_{k}^{\beta_k-1}z_{k+1}^{\beta_{k+1}}\ldots z_l^{\beta_l}\ldots z_n^{\beta_n}\right)\nonumber\\&+\alpha 
\Re\left(\sum\limits_{m=2}^{\infty}\sum\limits_{|\beta|=m}\beta_{jk} a_{\beta}z_1^{\beta_1}\ldots z_{j-1}^{\beta_{j-1}}z_{j}^{\beta_j^*}z_{j+1}^{\beta_{j+1}}\ldots z_{k-1}^{\beta_{k-1}}z_{k}^{\beta_k^{**}}z_{k+1}^{\beta_{k+1}}\ldots z_l^{\beta_l+1}\ldots z_n^{\beta_n}\right)\nonumber\\\geq &
1-\left|\sum\limits_{m=2}^{\infty}\sum\limits_{|\beta|=m}\beta_{k} a_{\beta}z_1^{\beta_1}\ldots z_{j-1}^{\beta_{j-1}}z_{j}^{\beta_j}z_{j+1}^{\beta_{j+1}}\ldots z_{k-1}^{\beta_{k-1}}z_{k}^{\beta_k-1}z_{k+1}^{\beta_{k+1}}\ldots z_l^{\beta_l}\ldots z_n^{\beta_n}\right|\nonumber\\&-
\alpha 
\left|\sum\limits_{m=2}^{\infty}\sum\limits_{|\beta|=m}\beta_{jk} a_{\beta}z_1^{\beta_1}\ldots z_{j-1}^{\beta_{j-1}}z_{j}^{\beta_j^*}z_{j+1}^{\beta_{j+1}}\ldots z_{k-1}^{\beta_{k-1}}z_{k}^{\beta_k^{**}}z_{k+1}^{\beta_{k+1}}\ldots z_l^{\beta_l+1}\ldots z_n^{\beta_n}\right|\nonumber\\\geq &
1-\sum\limits_{m=2}^{\infty}\sum\limits_{|\beta|=m}\beta_{k} \left|a_{\beta}\right|-\alpha 
\sum\limits_{m=2}^{\infty}\sum\limits_{|\beta|=m}\beta_{jk} \left|a_{\beta}\right|\nonumber\\\geq&
1-\sum\limits_{m=2}^{\infty}\sum\limits_{|\beta|=m}\sum\limits_{l=1}^n\sum\limits_{k=1}^n\sum\limits_{j=1}^n\left(\beta_{k}+\alpha \beta_{jk}\right)|a_{\beta}|\nonumber
\end{align}
for all $j,k,l\in\mathbb{Z}[1,n]$.
Using (\ref{Th3:1.10}) together with (\ref{Th5:1.1}), we get	
\begin{align}\label{Th5:1.2}
\Re\left(\frac{\partial h(z)}{\partial z_k}+\alpha z_l\frac{\partial^2 h(z)}{\partial z_j\partial z_k}\right)\geq 1-\sum\limits_{m=2}^{\infty}\sum\limits_{|\beta|=m}n\left(m^2\alpha+(n-\alpha)m\right)|a_{\beta}|
\end{align}	
for all $j,k,l\in\mathbb{Z}[1,n]$.

Again from (\ref{Th2:1.5a}) and (\ref{Th2:1.5}), we deduce that
\begin{align}\label{Th5:1.3}
&\left|\frac{\partial g(z)}{\partial z_k}+\alpha z_l\frac{\partial^2 g(z)}{\partial z_j\partial z_k}\right|\\\leq&
\sum\limits_{m=2}^{\infty}\sum\limits_{|\beta|=m}\beta_{k} \left|b_{\beta}z_1^{\beta_1}\ldots z_{j-1}^{\beta_{j-1}}z_{j}^{\beta_j}z_{j+1}^{\beta_{j+1}}\ldots z_{k-1}^{\beta_{k-1}}z_{k}^{\beta_k-1}z_{k+1}^{\beta_{k+1}}\ldots z_l^{\beta_l}\ldots z_n^{\beta_n}\right|\nonumber\\&+\alpha 
\sum\limits_{m=2}^{\infty}\sum\limits_{|\beta|=m}\beta_{jk} \left|b_{\beta}z_1^{\beta_1}\ldots z_{j-1}^{\beta_{j-1}}z_{j}^{\beta_j^*}z_{j+1}^{\beta_{j+1}}\ldots z_{k-1}^{\beta_{k-1}}z_{k}^{\beta_k^{**}}z_{k+1}^{\beta_{k+1}}\ldots z_l^{\beta_l+1}\ldots z_n^{\beta_n}\right|\nonumber
\\\leq &
\sum\limits_{m=2}^{\infty}\sum\limits_{|\beta|=m}\beta_{k} \left|b_{\beta}\right|+\alpha 
\sum\limits_{m=2}^{\infty}\sum\limits_{|\beta|=m}\beta_{jk} \left|b_{\beta}\right|\nonumber\\\leq&
\sum\limits_{m=2}^{\infty}\sum\limits_{|\beta|=m}n\left(m^2\alpha+(n-\alpha)m\right)|b_{\beta}|\nonumber
\end{align}
for all $j,k,l\in\mathbb{Z}[1,n]$.
Now using (\ref{Th5:1.0}), (\ref{Th5:1.2}) and (\ref{Th5:1.3}), we get
\begin{align*}
\Re\left(\frac{\partial h(z)}{\partial z_k}+\alpha z_l\frac{\partial^2 h(z)}{\partial z_j\partial z_k}\right)>&\sum\limits_{m=2}^{\infty}\sum\limits_{|\beta|=m}n\left(m^2\alpha+(n-\alpha)m\right)|b_{\beta}|\\\geq&
\left|\frac{\partial g(z)}{\partial z_k}+\alpha z_l\frac{\partial^2 g(z)}{\partial z_j\partial z_k}\right|
\end{align*}
for all $j,k,l\in\mathbb{Z}[1,n]$ and so $f\in \mathcal{W}_{\mathcal{H}_n^0}(\alpha)$.
\end{proof}

Finally, we show that $\mathcal{W}_{\mathcal{H}_n^0}(\alpha)$ is closed under convex combinations.

\begin{theo}\label{Th-1.6} The class $\mathcal{W}_{\mathcal{H}_n^0}(\alpha)$ is closed under convex combination.
\end{theo}
\begin{proof} Let $f_s=h_s+\ol g_s\in \mathcal{W}_{\mathcal{H}_n^0}(\alpha)$ for $s=1,2,\ldots,m$ and $\sum_{s=1}^{m} t_s=1\;(0\leq t_s\leq 1)$.
The convex combination of the functions $f_s$'s can be written as
\begin{align*}
f(z)=\sum\limits_{s=1}^m t_sf_s=h(z)+\ol {g(z)},
\end{align*}
where $h(z)=\sum_{s=1}^m t_s h_s(z)$ and $g(z)=\sum_{s=1}^m t_s g_s(z)$. Then both $h(z)$ and $g(z)$ are holomorphic in $\mathbb{P} \Delta(0;1)$ with $h(0)=g(0)=0$. Then	\begin{align*}
\;\left(\frac{\partial h(0)}{\partial z_1},\ldots,\frac{\partial h(0)}{\partial z_n}\right)=(1,\ldots,1)\;\text{and}\; \left(\frac{\partial g(0)}{\partial z_1},\ldots,\frac{\partial g(0)}{\partial z_n}\right)=(0,\ldots,0).
\end{align*}
	
A simple computation shows that
\begin{align*}
\Re\left(\frac{\partial h(z)}{\partial z_k}+\alpha z_l\frac{\partial^2 h(z)}{\partial z_j\partial z_k}\right)=&\Re\left(\sum\limits_{s=1}^mt_s\left(\frac{\partial h_s(z)}{\partial z_k}+\alpha z_l\frac{\partial^2 h_s(z)}{\partial z_j\partial z_k}\right)\right)\\>&
\sum\limits_{s=1}^m t_s\left|\frac{\partial g_s(z)}{\partial z_k}+\alpha z_l\frac{\partial^2 g_s(z)}{\partial z_j\partial z_k} \right|\\\geq&
\left|\frac{\partial g(z)}{\partial z_k}+\alpha z_l\frac{\partial^2 g(z)}{\partial z_j\partial z_k} \right|
\end{align*}
for all $j,k,l\in\mathbb{Z}[1,n]$ and so $f\in \mathcal{W}_{\mathcal{H}_n^0}(\alpha)$.
\end{proof}

\section{{\bf Partial sums of functions in $\mathcal{W}_{\mathcal{H}_n^0}(\alpha)$}}
Let 
\begin{align*}
f(z)=\sum\limits_{k=1}^n z_k+\sum\limits_{m=2}^{\infty}\sum\limits_{|\beta|=m}a_{\beta}z^{\beta}
\end{align*}
in $\in \mathbb{P}\Delta(0;1_n)$. Then the $p$-th partial sum (or section) of $f(z)$ is defined by
\begin{align*}
S_p(f(z))=\sum\limits_{k=1}^n z_k+\sum\limits_{m=2}^{p}\sum\limits_{|\beta|=m}a_{\beta}z^{\beta}
\end{align*}
in $\in \mathbb{P}\Delta(0;1_n)$ and $p\geq 2$. Analogously in the pluriharmonic case, the $p, q$-th partial sum (or section) of a harmonic function $f=h + \ol g$ given by (\ref{Eq 1.7}) is defined by
\begin{align*}
S_{p,q}(f)=S_p(h)+\ol{S_q(g)},
\end{align*}
where
\begin{align*}
S_p(h(z))=\sum\limits_{k=1}^n z_k+\sum\limits_{m=2}^{p}\sum\limits_{|\beta|=m}a_{\beta}z^{\beta}\;\;\text{and}\;\;S_q(g(z))=\sum\limits_{m=2}^{q}\sum\limits_{|\beta|=m}b_{\beta}z^{\beta}
\end{align*}
in $\in \mathbb{P}\Delta(0;1_n)$ and $p,q\geq 2$.\vspace{1.2mm}

In 1928 Szeg\"{o} \cite{Szego-1928} proved the remarkable result that every section $S_p(f)$ of a function $f\in\mathcal{S}$ is univalent in the disk $|z|<1/4$. That the number $1/4$ is best possible is evident from the second partial sum of the Koebe function $k(z)= z/(1-z)^2$.\vspace{1.2mm}

The study of sections of harmonic mappings is motivated by the significant interest in approximating real-valued harmonic functions by harmonic polynomials (see \cite{Walsh-1929}), owing to the many advantages such approximations offer. For instance, a harmonic function attains its maximum and minimum values on the boundary of the domain under consideration. Since planar harmonic mappings $f=h+\ol g$
defined on $\mathbb{D}$ admit a series representation, their sections can be viewed as approximations by complex-valued harmonic polynomials. Consequently, the approximation of univalent harmonic mappings by univalent harmonic polynomials may provide new insights and potential applications, particularly in problems arising in fluid dynamics and fluid flow theory.\vspace{1.2mm}

In $2013$, Li and Ponnusamy \cite{Li-Ponnusamy-2013} discussed the sections of functions in the class
\begin{align*}
\mathcal{P}^{0}_{\mathcal{H}}(\alpha)
= \left\lbrace f=h+\overline{g}\in\mathcal{H} :
	\Re\bigl(h'(z)-\alpha\bigr) > |g'(z)|,\quad z\in\mathbb{D}
	\right\rbrace
\end{align*}
and in 2015, Ponnusamy et al. \cite{Li-Ponnusamy-2015} also investigated properties of the sections of stable harmonic convex functions (see also \cite{Ponnusamy-Kaliraj-Starkov-2017}). For several other interesting results concerning sections of analytic functions, we refer the reader to \cite{Obradovic-Ponnusamy-2013, Obradovic-Ponnusamy-2014, Ponnusamy-Sahoo-Yanagihara-2014,Silverman-1988, Singh-1970}.\vspace{1.2mm}

Regarding sections of harmonic mappings, Ghosh and Vasudevarao \cite{Ghosh-Allu-2019} obtained the following result.

\begin{theoG}\emph{\cite[Theorem 6.2]{Ghosh-Allu-2019}}
Let $f \in \mathcal{W}_{\mathcal{H}}^{0}(\alpha)$. Then, for each $q \geq 2$, $S_{1,q}(f)\in \mathcal{W}_{\mathcal{H}}^{0}(\alpha)$ for $|z|< \frac{1}{2}$.
\end{theoG}

In this section, we introduce the sections of pluriharmonic mappings in the unit polydisk $\mathbb{P}\Delta(0;1_n)$ and obtained the following result.

\begin{theo}\label{Th7-1} Let $f=h+\ol g\in \mathcal{W}_{\mathcal{H}_n^0}(\alpha)$ with $\alpha\geq 0$ and be given by (\ref{Eq 1.7}). Then for each $q \geq 2$ and $\|z\|_{\infty}< \frac{1}{2n^5}$; $S_{1,q}(f)\in \mathcal{W}_{\mathcal{H}_n^0}(\alpha)$.	
\end{theo}
\begin{proof}
Let $f=h+\ol g\in \mathcal{W}_{\mathcal{H}_n^0}(\alpha)$, where $h$ and $g$ are given by (\ref{Eq 1.7}). Note that
\begin{align*}
	S_{1,q}(f(z))=S_1(h(z))+\ol{S_q(g(z))}=\sum\limits_{k=1}^n z_k+\ol{\sum\limits_{m=2}^q\sum_{|\beta|=m}b_{\beta}z^{\beta}}.
\end{align*}

Observe that
\begin{align*}
	\frac{\partial S_1(h(z))}{\partial z_k}+\alpha z_l
	\frac{\partial^2 S_1(h(z))}{\partial z_j\partial z_k}=1
\end{align*}
and so 
\begin{align*}
\Re\left(\frac{\partial S_1(h(z))}{\partial z_k}+\alpha z_l
	\frac{\partial^2 S_1(h(z))}{\partial z_j\partial z_k}\right)=1
\end{align*}
for all $j,k,l\in\mathbb{Z}[1,n]$.
On the other hand, from (\ref{Th2:1.5a}) and (\ref{Th2:1.5}), we deduce that
\begin{align}\label{Th7:1.1}
&\left|\frac{\partial S_q(g(z))}{\partial z_k}+\alpha z_l\frac{\partial^2 S_q(g(z))}{\partial z_j\partial z_k}\right|\\\leq&
\sum\limits_{m=2}^{q}\sum\limits_{|\beta|=m}\beta_{k} \left|b_{\beta}z_1^{\beta_1}\ldots z_{j-1}^{\beta_{j-1}}z_{j}^{\beta_j}z_{j+1}^{\beta_{j+1}}\ldots z_{k-1}^{\beta_{k-1}}z_{k}^{\beta_k-1}z_{k+1}^{\beta_{k+1}}\ldots z_l^{\beta_l}\ldots z_n^{\beta_n}\right|\nonumber\\&+\alpha 
\sum\limits_{m=2}^{q}\sum\limits_{|\beta|=m}\beta_{jk} \left|b_{\beta}z_1^{\beta_1}\ldots z_{j-1}^{\beta_{j-1}}z_{j}^{\beta_j^*}z_{j+1}^{\beta_{j+1}}\ldots z_{k-1}^{\beta_{k-1}}z_{k}^{\beta_k^{**}}z_{k+1}^{\beta_{k+1}}\ldots z_l^{\beta_l+1}\ldots z_n^{\beta_n}\right|\nonumber
\\\leq &
\sum\limits_{m=2}^{q}\sum\limits_{|\beta|=m}\beta_{k} \left|b_{\beta}\right|||z||_{\infty}^{m-1}+\alpha 
\sum\limits_{m=2}^{\infty}\sum\limits_{|\beta|=m}\beta_{jk} \left|b_{\beta}\right|||z||_{\infty}^{m-1}\nonumber\\\leq&
\sum\limits_{m=2}^{q}\sum\limits_{|\beta|=m}\sum\limits_{l=1}^n\sum\limits_{k=1}^{n}\sum\limits_{j=1}^n\left(\beta_{k}+\alpha \beta_{jk}\right)|b_{\beta}|||z||_{\infty}^{m-1}\nonumber\\=&
\sum\limits_{m=2}^{q}\sum\limits_{|\beta|=m}n\left(m^2\alpha+(n-\alpha)m\right)|b_{\beta}|||z||_{\infty}^{m-1}\nonumber.
\end{align}
Using Theorem \ref{Th-1.2} together with \eqref{Th7:1.1}, we obtain
\begin{align*}
 \left|\frac{\partial S_q(g(z))}{\partial z_k}+\alpha z_l\frac{\partial^2 S_q(g(z))}{\partial z_j\partial z_k}\right|\leq& \sum\limits_{m=2}^q \binom{m+n-1}{n-1}n^3||z||_{\infty}^{m-1}\\\leq&
 \sum\limits_{m=2}^q n^{m+3}||z||_{\infty}^{m-1}\\<&
 \sum\limits_{m=2}^q  \frac{n^{m+3}}{2^{m-1}n^{5(m-1)}}\\\leq&
 \sum\limits_{m=2}^q \frac{1}{2^{m-1}}\\<& 1=\Re\left(\frac{\partial S_1(h(z))}{\partial z_k}+\alpha z_l
 \frac{\partial^2 S_1(h(z))}{\partial z_j\partial z_k}\right)
 \end{align*}
 for all $j,k,l\in\mathbb{Z}[1,n]$. Therefore $S_{1,q}(f)\in \mathcal{W}_{\mathcal{H}_n^0}(\alpha)$.
\end{proof}

\vspace{5mm}

\noindent\textbf{Conflict of interest:} The authors declare that there is no conflict  of interest regarding the publication of this paper.\vspace{1.2mm}

\noindent {\bf Funding:} Not Applicable.\vspace{1.2mm}

\noindent\textbf{Data availability statement:}  Data sharing not applicable to this article as no datasets were generated or analysed during the current study.\vspace{1.2mm}

\noindent {\bf Authors' contributions:} All the authors have equal contributions in preparation of the manuscript.

\end{document}